%% file: rescaled-polya-succ_ARXIV3.tex
\documentclass[10pt,a4paper]{article}

\usepackage{bbm}
\usepackage{amsmath}
\usepackage{amsfonts}
\usepackage{amsthm}
\usepackage{amssymb}
\usepackage{graphicx}
\usepackage[utf8]{inputenc}
\usepackage{color}
\allowdisplaybreaks

\numberwithin{equation}{section}

\theoremstyle{remark}
\newtheorem{remark}{Remark}[section]

\theoremstyle{plain}
\newtheorem{theorem}[remark] {Theorem}
\newtheorem{proposition}[remark]{Proposition}
\newtheorem{lemma}[remark]{Lemma}

\theoremstyle{definition}

\newtheorem{example}[remark]{Example}

\newcommand{\R}{\mathbb R}

\newcommand{\bs}[1]{{\boldsymbol{#1}}}

\begin{document}

\title{Generalized Rescaled P\'olya urn and its statistical applications} 
\author{ Giacomo Aletti
\footnote{ADAMSS Center,
  Universit\`a degli Studi di Milano, Milan, Italy, giacomo.aletti@unimi.it}
$\quad$and$\quad$ 
Irene Crimaldi
\footnote{IMT School for Advanced Studies, Lucca, Italy, 
irene.crimaldi@imtlucca.it}
} 
\maketitle

\abstract{We introduce the Generalized Rescaled P\'olya (GRP) urn, that
provides a generative model for a chi-squared test of goodness of fit
for the long-term probabilities of clustered data, with independence
between clusters and correlation, due to a reinforcement mechanism,
inside each cluster.  We apply the proposed test to a data set of
Twitter posts about COVID-19 pandemic: in a few words, for a classical
chi-squared test the data result strongly significant for the
rejection of the null hypothesis (the daily long-run sentiment rate
remains constant), but, taking into account the correlation among
data, the introduced test leads to a different conclusion. Beside the
statistical application, we point out that the GRP urn is a simple
variant of the standard Eggenberger-P\'olya urn, that, with suitable
choices of the parameters, shows ``local'' reinforcement, almost sure
convergence of the empirical mean to a deterministic limit and
different asymptotic behaviours of the predictive mean. Moreover, the
study of this model provides the opportunity to analyze stochastic
approximation dynamics, that are unusual in the related literature.  \\[10pt]
}
\\[10pt]
 
 \renewcommand\contentsname{\noindent\hrulefill\\ Contents}
\tableofcontents
% \addtocontents{toc}{\hrulefill\par}
\addtocontents{toc}{~\hfill\textbf{Page}\par}

\section{Introduction}%\label{intro}

The standard Eggenberger-P\'olya urn (see  	{EggPol23,mah}) has
been widely studied and generalized (for instance, some recent
variants can be found in
\cite{AlGhRo,AlGhVi,BeCrPrRi16,chen2013,collevecchio2013,lar-pag,mailler,mah}). 
In its simplest form, this model with $k$-colors works as follows.
An urn contains $N_{0\, i}$ balls of color $i$, for $i=1,\dots, k$,
and, at each discrete time, a ball is extracted from the urn and then
it is returned inside the urn together with $\alpha>0$ additional
balls of the same color. Therefore, if we denote by $N_{n\, i}$ the
number of balls of color $i$ in the urn at time $n$, we have
\begin{equation*}
N_{n\, i}=N_{n-1\,i}+\alpha\xi_{n\,i}\qquad\mbox{for } n\geq 1,
\end{equation*}
where $\xi_{n\,i}=1$ if the extracted ball at time $n$ is of color
$i$, and $\xi_{n\,i}=0$ otherwise. The parameter $\alpha$ regulates
the reinforcement mechanism: the greater $\alpha$, the greater the
dependence of $N_{n\,i}$ on $\sum_{h=1}^n\xi_{h\,i}$.\\
\indent The Generalized Rescaled P\'olya (GRP) urn model is
characterized by the introduction of the sequence $(\beta_n)_n$ of
parameters, together with the replacement of the parameter $\alpha$ of
the original model by a sequence $(\alpha_n)_n$, so that
\begin{equation*}%\label{eq-dynamics-intro}
\begin{aligned}
N_{n\, i}& =b_{0\, i}+B_{n\, i} &&\text{with }
\\
B_{n+1\, i}&=\beta_n B_{n\, i}+\alpha_{n+1}\xi_{n+1\, i}&& n\geq 0.
\end{aligned}
\end{equation*}
Therefore, the urn initially contains $b_{0\,i}+B_{0\,i}$ balls of
color $i$ and the parameters $\beta_n\geq 0$, together with
$\alpha_n>0$, regulate the reinforcement mechanism. More precisely,
the term $\beta_n B_{n\,i}$ links $N_{n+1\,i}$ to the
``configuration'' at time $n$ through the ``scaling'' parameter
$\beta_n$, and the term $\alpha_{n+1}\xi_{n+1\,i}$ links $N_{n+1\,i}$
to the outcome of the extraction at time $n+1$ through the parameter
$\alpha_{n+1}$.\\
\indent We are going to show that, with a suitable choice of the model
parameters, we have a long-term almost sure convergence of the
empirical mean $\sum_{n=1}^N\xi_{n\, i}/N$ to the deterministic limit
$p_{0\,i}=b_{0\,i}/\sum_{i=1}^nb_{0\,i}$, and a chi-squared goodness
of fit result for the long-term probabilities
$\{p_{0\,1},\dots,p_{0\,k}\}$. In particular, regarding the last
point, we have that the chi-squared statistics
\begin{equation}\label{eq:chi1_start}
  \chi^2 =
  N \sum_{i=1}^k \frac{( \widehat{p}_i - p_{0\,i})^2}{p_{0\,i}}=
  \sum_{i=1}^k \frac{(O_i-Np_{0\,i})^2}{Np_{0\,i}},
\end{equation}
where $N$ is the size of the sample, $\widehat{p}_i=O_i/N$, with
$O_i=\sum_{n=1}^N\xi_{n\, i}$ the number of observations equal to $i$
in the sample, is asymptotically distributed as $\chi^2(k-1)\lambda$,
with $\lambda>1$, or $\chi^2(k-1)N^{1-2e}\lambda$, where $\lambda>0$
may be smaller than $1$, but $e$ is always strictly smaller than
$1/2$.  In both cases, the presence of correlation among units
mitigates the effect in \eqref{eq:chi1_start} of the sample size $N$,
that multiplies the chi-squared distance between the observed
frequencies and the expected probabilities. This aspect is important
for the statistical applications in the context of a ``big sample'',
when a small value of the chi-squared distance might be significant,
and hence a correction related to the correlation between observations
is desirable (see, for instance,
\cite{bergh,BERTONI2018,chanda,MR0403125,gleser,knoke2002,MR1894384,radlow,RaoScott81,MR3190613}). More
precisely, in the first case, the observed value of the chi-squared
distance has to be compared with the ``critical'' value
$\chi^2_{1-\theta}(k-1)\lambda/N$, where $ \chi^2_{1-\theta}(k-1)$
denotes the quantile of order $1-\theta$ of the chi-squared
distribution $\chi^2(k-1)$. In the second case, the critical value for
the chi-squared distance becomes
$\chi^2_{1-\theta}(k-1)\lambda/N^{2e}$, where, although the constant
$\lambda$ may be smaller than $1$, the effect of the sample size $N$
is mitigated by the exponent $2e<1$. In other words, for this second
case, the Fisher information given by the sample does not scale with
the sample size $N$, but with rate $N^{2e}$. Hence, since the
long-term correlation, collecting more and more data does not provide
a linear increment of the information.  \\ \indent Summing up, the GRP
urn provides a theoretical framework for a chi-squared test of
goodness of fit for the long-term probabilities of correlated data,
generated according to a reinforcement mechanism. Specifically, we
describe a possible application in the context of clustered data, with
independence between clusters and correlation, due to a reinforcement
mechanism, inside each cluster.  In particular, we develop a suitable
estimation technique for the fundamental model parameters. We then
apply the proposed test to a data set of Twitter posts about COVID-19
pandemic. Given the null hypothesis that the daily long-run sentiment
rate of the posts is the same for all the considered days (suitably
spaced days in the period February 20th - April 20th 2020), performing
a classical $\chi^2$ test, the data result strongly significant for
the rejection of the null hypothesis, while, taking into account the
correlation among posts sent in the same day, the proposed test leads
to a different conclusion.  \\
\indent The sequel of the paper is so structured. In Section
\ref{urn-model} we set up the notation and we define the GRP urn. In
Section \ref{sec-comparison} we illustrate its relationships with
previous models and we discuss the connections with related
literature. In particular, the object of the present work gives us the
opportunity to study Stochastic Approximation (SA) dynamics, which are
infrequent in SA literature and so fill in some theoretical gaps. In
Section \ref{sec-main} we provide the main result of this work, that
is the almost sure convergence of the empirical means to the
deterministic limits $p_{0\,i}$ and the goodness of fit result for the
long-term probabilities $p_{0\,i}$, together with comments and
examples. In Section \ref{sec-appl} we describe a possible statistical
application of the GRP urn and the related results: a chi-squared test
of goodness of fit for the long-term probabilities of clustered
data, with independence between clusters and correlation, due to a
reinforcement mechanism, inside each cluster.  We apply the proposed
test to a data set of Twitter posts about COVID-19 pandemic. In Section
\ref{sec-CLT} we state two convergence results for the empirical
means, which are the basis for the proof of the main theorem. All the
shown theoretical results are analytically proven. The proofs are left
to Section~\ref{sec-proofs} in the Supplementary Material
\cite{AleCri_new20_suppl}, except for the proof of Theorem~\ref{CLT},
which is methodologically new and emphasizes new techniques of
martingale limit theory and so it is illustrated in Section
\ref{proof-CLT}. Finally, in the Supplementary Material we also
provide some complements, some technical lemmas and some recalls about
stochastic approximation theory and about stable convergence.  When
necessary, the references to the Supplementary Material are preceded
by an ``S'', so that (S$1.2$) will refer to the equation (S$1.2$) in
\cite{AleCri_new20_suppl}.

%%%%%%%%%%%%%%%%%%%%%%%%%%%%%%%%%%%%%%%%%%%%%%%%

\section{The Generalized Rescaled P\'olya (GRP) urn}
\label{urn-model}

In all the sequel, we suppose given two sequences of parameters
$(\alpha_n)_{n\geq 1}$, with $\alpha_n>0$ and $(\beta_n)_{n\geq 0}$
with $\beta_n \geq 0$.  Given a vector $\bs{x}= (x_1, \ldots,
x_k)^\top\in \mathbb{R}^k$, we set \( |\bs{x}| = \sum_{i=1}^k |x_i| \)
and \( \|\bs{x}\|^2 = \bs{x}^\top \bs{x}= \sum_{i=1}^k |x_i|^2 \).
Moreover we denote by $\bs{1}$ and $\bs{0}$ the vectors with all the
components equal to $1$ and equal to $0$, respectively.  \\

\indent The urn initially contains $b_{0\,i}+B_{0\,i}>0$ distinct balls
of color $i$, with $i=1,\dots,k$.  We set
$\bs{b_0}=(b_{0\,1},\dots,b_{0\,k})^{\top}$ and
$\bs{B_0}=(B_{0\,1},\dots,B_{0\,k})^{\top}$. We assume $|\bs{b_0}|>0$
and we set $\bs{p_0} = \frac{\bs{b_0}}{|\bs{b_0}|}$.  At each discrete
time $(n+1)\geq 1$, a ball is drawn at random from the urn, obtaining
the random vector $\bs{\xi_{n+1}} = (\xi_{n+1\,1}, \ldots,
\xi_{n+1\,k})^\top$ defined as
\begin{equation*}
\xi_{n+1\,i} = 
\begin{cases}
1  &  \text{when the extracted ball at time $n+1$ is of color $i$}
\\
0  & \text{otherwise},
\end{cases}
\end{equation*}
and the number of balls in the urn is so updated:
\begin{equation}\label{eq:reinf1:K}
\bs{N_{n+1}}=\bs{b_0}+\bs{B_{n+1}}\qquad\text{with}
\qquad
\bs{B_{n+1}} = \beta_n \bs{B_n} + \alpha_{n+1} \bs{\xi_{n+1}}\,,
\end{equation}
which gives (since $|\bs{\xi_{n+1}}|=1$)
\begin{equation*}%\label{eq:reinf1:K-bis}
|\bs{B_{n+1}}|= \beta_n |\bs{B_n}| + \alpha_{n+1}.  
\end{equation*}
Therefore, setting $r^*_{n} = |\bs{N_n}|= |\bs{b_0}|+|\bs{B_n}|$, we
get
\begin{equation}\label{dinamica-rnstar}
r_{n+1}^*=r_n^*+(\beta_n-1)|\bs{B_n}|+\alpha_{n+1}, 
\end{equation}
that is 
\begin{equation}\label{dinamica-rnstar2}
r_{n+1}^*-r_n^*=|\bs{b_0}| (1-\beta_n) - r_n^* (1-\beta_n) +\alpha_{n+1}.
\end{equation}
Moreover, setting $\mathcal{F}_0$ equal to the trivial $\sigma$-field
and $\mathcal{F}_n=\sigma(\bs{\xi_1},\dots,\bs{\xi_n})$ for $n\geq 1$,
the conditional probabilities $\bs{\psi_{n}}= (\psi_{n\,1}, \ldots,
\psi_{n\,k})^\top$ of the extraction process, also called {\em predictive
means}, are
\begin{equation}\label{eq:extract1a:K-vettoriale}
\bs{\psi_{n}}=E[ \bs{\xi_{n+1}}| \mathcal{F}_n ]= 
\frac{\bs{N_n}}{|\bs{N_n}|} =
\frac{\bs{b_0}+\bs{B_n}}{r_n^*}\qquad n\geq 0.
\end{equation}
It is obvious that we have $|\bs{\psi_{n}}|=1$. Moreover, when
$\beta_n>0$ for all $n$, the probability $\psi_{n\,i}$ results
increasing with the number of times we observed the value $i$, that is
the random variables $\xi_{n\,i}$ are generated according to a
reinforcement mechanism: the probability that the extraction of color
$i$ occurs has an increasing dependence on the number of extractions
of color $i$ occurred in the past (see,
e.g. \cite{pemantle2007}). More precisely, we have
\begin{equation}\label{eq-psi_n}
\bs{\psi_n} = 
\frac{ 
\bs{b_0}+ \bs{B_0}\prod_{j=0}^{n-1} \beta_j + 
\sum_{h=1}^n \left(\alpha_h \prod_{j=h}^{n-1}\beta_j \right)\bs{\xi_{h}}
}
{
|\bs{b_0}|+ |\bs{B_0}|\prod_{j=0}^{n-1}\beta_j +
\sum_{h=1}^n  \left(\alpha_h \prod_{j=h}^{n-1}\beta_j \right)
}\,.
\end{equation}
The dependence of $\bs{\psi_n}$ on $\bs{\xi_h}$ depends on the factor
$f(h,n)=\alpha_h \prod_{j=h}^{n-1}\beta_j$, with $1\leq h\leq
n,\,n\geq 0$. In the case of the standard Eggenberger-P\'olya urn,
that corresponds to $\alpha_n=\alpha>0$ and $\beta_n=1$ for all $n$,
each observation $\bs{\xi_h}$ has the same ``weight''
$f(h,n)=\alpha$. Instead, if the factor $f(h,n)$ increases with $h$,
then the main contribution is given by the most recent extractions. We
refer to this phenomenon as ``local'' reinforcement. For instance,
this is the case when $(\alpha_n)$ is increasing and $\beta_n=1$ for
all $n$. Another case is when $\alpha_n=\alpha>0$ and $\beta_n<1$ for
all $n$. The case $\beta_n=0$ for all $n$ is an extreme case, for
which $\bs{\psi_n}$ depends only on the last extraction $\bs{\xi_n}$
(recall that conventionally $\prod_{j=n}^{n-1}=1$). For the next
examples, we will show that they exhibit a broader sense local
reinforcement, in the sense that the ``weight'' of the observations is
eventually increasing with time.  \\ \indent By means of
\eqref{eq:extract1a:K-vettoriale}, together with \eqref{eq:reinf1:K}
and \eqref{dinamica-rnstar}, we have
\begin{equation}\label{dinamica-psin}
\bs{\psi_{n+1}}-\bs{\psi_{n}}=
-\frac{(1-\beta_{n})}{r_{n+1}^*}|\bs{b_0}|\big(\bs{\psi_{n}}-\bs{p_0}\big)
+
\frac{\alpha_{n+1}}{r_{n+1}^*}\big(\bs{\xi_{n+1}}-\bs{\psi_{n}}\big).
\end{equation}
\indent Setting $\bs{\theta_n} = \bs{\psi_{n}}-\bs{p_0}$ and $\Delta
\bs{M_{n+1}}= \bs{\xi_{n+1}}-\bs{\psi_{n}} = \bs{\xi_{n+1}}-\bs{p_0} -
\bs{\theta_n}$ and letting $\epsilon_n =|\bs{b_0}|
\frac{(1-\beta_{n})}{r_{n+1}^*}$ and $\delta_n
=\alpha_{n+1}/r_{n+1}^*$, from \eqref{dinamica-psin} we obtain
\begin{equation}\label{dinamica-psin-bis}
\bs{\psi_{n+1}}-\bs{\psi_{n}}=-\epsilon_{n}(\bs{\psi_{n}}-\bs{p_0})+
\delta_n\Delta\bs{M_{n+1}}
\end{equation}
and so 
\begin{equation}\label{dinamica-thetan}
\bs{\theta_{n+1}}-\bs{\theta_{n}}=-\epsilon_{n}\bs{\theta_{n}}
+\delta_n\Delta\bs{M_{n+1}}\,.
\end{equation}
Therefore, the asymptotic behaviour of $(\bs{\theta_n})$ depends on
the two sequences $(\epsilon_n)_n$ and $(\delta_n)_n$.  
\\
\indent Finally, we observe that, setting
$\bs{\overline{\xi}_{N}}=\sum_{n=1}^N\bs{\xi_{n}}/N$ and
$\bs{\mu_n}=\bs{\overline{\xi}_n}-\bs{p_0}$, we have the equality
\begin{equation}\label{dinamica-mun}
\bs{\mu_{n+1}}-\bs{\mu_n}=-\frac{1}{n}(\bs{\mu_n}-\bs{\theta_n})
+\frac{1}{n}\Delta\bs{M_{n+1}},
\end{equation}
that links the asymptotic behaviour of $(\bs{\mu_n})$ and the one of
$(\bs{\theta_n})$.  \\ \indent Different kinds of sequences
$(\epsilon_n)_n$ and $(\delta_n)_n$ provide different kinds of
asymptotic behaviour of $\bs{\theta_n}$, i.e.~of the empirical mean
$\bs{\overline{\xi}_{N}}$. In Section \ref{sec-main}, we provide two
cases in which we have a long-term almost sure convergence of the
empirical mean $O_i/N=\sum_{n=1}^N\xi_{n\, i}/N$ toward the constant
$p_{0i}=b_{0\,i}/|\bs{b_0}|$, together with a chi-squared goodness of
fit result. In particular, the quantities $p_{0\,1},\ldots,p_{0\,k}$
can be seen as a long-run probability distribution on the possible
values (colors) $\{1,\dots,k\}$. 

%%%%%%%%%%%%%%%%%%%%%%%%%%%%%%%%%%%%%%%%%%%%%%%%%%%%%%%%%%
\section{Related literature}\label{sec-comparison}

The particular case when in the GRP urn model we have
$\beta_n=\beta=0$ for all $n$ corresponds to a version of the
so-called ``memory-1 senile reinforced random walk'' on a star-shaped
graph introduced in \cite{holmes}. The case $\alpha_n=\alpha>0$ and
$\beta_n=\beta=1$ for all $n$ corresponds to the standard
Eggenberger-P\'olya urn with an initial number
$N_{0\,i}=b_{0\,i}+B_{0\,i}$ of balls of color $i$.  When $(\alpha_n)$
is a not-constant sequence, while $\beta_n=\beta=1$ for all $n$, the
GRP urn coincides with the variant of the Eggenberger-P\'olya urn
introduced in \cite{pem90} (see also \cite[Sec. 3.2]{pemantle2007}).
Instead, when $\beta\neq 1$, the GRP urn does not fall in any variants
of the Eggenberger-P\'olya urn discussed in
\cite[Sec. 3.2]{pemantle2007}. \\ \indent The case when
$\alpha_n=\alpha>0$ and $\beta_n=\beta\geq 0$ for all $n$ corresponds
to the Rescaled P\'olya (RP) urn introduced and studied in
\cite{ale-cri-RP} and applied in \cite{ale-cri-sar}. It is worthwhile
to point out that the two cases studied in the present work do not
include (and are not included in) the case studied in
\cite{ale-cri-RP}.  Moreover, the techniques employed here and in
\cite{ale-cri-RP} are completely different: when $\beta_n=\beta\in
     [0,1)$ as in \cite{ale-cri-RP}, the jumps $\Delta\bs{\psi_n}$ do
       not vanish and the process $\bs{\psi}=(\bs{\psi_n})_n$
       converges to a stationary Markov chain and so the appropriate
       Markov ergodic theory is employed; in this work, we have
       $|\Delta\bs{\psi_n}| = o(1)$, so that the martingale limit
       theory is here exploited to achieve the asymptotic
       results. Obviously, the two techniques are not exchangeable or
       adaptable from one contest to the other one.  \\ \indent When
       $(\beta_n)$ is not identically equal to $1$, since the first
       term in the right hand of the above relation, the GRP urn does
       not belong to the class of Reinforced Stochastic Processes
       (RSPs) studied in
       \cite{AlCrGh,ale-cri-ghi-WEIGHT-MEAN,ale-cri-ghi-MEAN,cri-dai-lou-min,cri-dai-min,dai-lou-min}. Indeed,
       the RSPs are characterized by a ``strict'' reinforcement
       mechanism such that $\xi_{n\,i}=1$ implies
       $\psi_{n\,i}>\psi_{n-1\,i}$ and so, as a consequence,
       $\psi_{n\,i}$ has an increasing dependence on the number of
       times we have $\xi_{h\,i}=1$ for $h=1,\dots,n$. When
       $(\beta_n)$ is not identically equal to $1$, the GRP urn does
       not satisfy the ``strict'' reinforcement mechanism, because the
       first term is positive or negative according to the sign of
       $(1-\beta_n)$ and of $(\bs{\psi_{n}}-\bs{p_0})$. Furthermore,
       we observe that equation \eqref{dinamica-psin} recalls the
       dynamics of a RSP with a ``forcing input'' (see
       \cite{AlCrGh,cri-dai-lou-min,sah}), but the main difference
       relies on the fact that such a process is driven by a classical
       stochastic approximation dynamics, that is a dynamics of the
       kind \eqref{dinamica-psin-bis} with $\epsilon_n=\delta_n$ (up
       to a constant) with $\sum_n\epsilon_n=+\infty$ and
       $\sum_n\epsilon_n^2<+\infty$, while the GRP urn model also
       allows for $\epsilon_n$ and $\delta_n$ with different rates and
       also for
\begin{itemize}
\item $\sum_n \epsilon_n=+\infty$ and $\sum_n \delta_n^2=+\infty$ or 
\item $\sum_n \epsilon_n<+\infty$.
\end{itemize}
\indent Since \eqref{dinamica-psin-bis} is the fundamental equation of
the Stochastic Approximation (SA) theory, we deem it appropriate to
say a few more words on the relationship of the present work with the
SA literature. The case when $\delta_n=c\epsilon_n$ in
\eqref{dinamica-psin-bis} is essentially covered by the Stochastic
Approximation (SA) theory (see Section~\ref{SA-app}, where we refer
to \cite{fort,KusYin03,mok-pel2006,pel,Zhang2016}). The most known
case is when $\sum_n\epsilon_n=+\infty$ and $\sum_n
\epsilon_n^2<+\infty$.  The case $\epsilon_n\to 0$,
$\sum_n\epsilon_n=+\infty$ and $\sum_n \epsilon_n^2=+\infty$ is less
usual in literature, but it is well characterized in \cite{KusYin03}.
The case when $(\epsilon_n)_n$ and $(\delta_n)_n$ in
\eqref{dinamica-psin-bis} go to zero with different rates is typically
neglected in SA literature. To our best knowledge, it is taken into
consideration only in \cite{pel}, where the weak convergence rate of
the sequence $(\bs{\psi_N})$ toward a certain point $\bs{\psi^*}$ is
established under suitable assumptions, given the event
$\{\bs{\psi_N}\to\bs{\psi^*}\}$. No result is given for the empirical
mean $\bs{\overline{\xi}}_N$, which instead is the focus of the
present paper (see Theorem \ref{th-chi-squared-test} below, whose
proof is based on Theorem \ref{CLT}). More precisely, the assumptions
on $\epsilon_n$ and $\delta_n$ in the following Theorem \ref{CLT}
imply assumption (A1.3) in \cite{pel} and so Theorem~1 in that paper
provides the weak convergence rate of the sequence
$(\bs{\psi_N}-\bs{\psi^*})$ given the event
$\{\bs{\psi_N}\to\bs{\psi^*}\}$. However, this result is not useful
for our scope because of two reasons: first, we need convergence
results for the empirical mean $\bs{\overline{\xi}_N}$, not for the
predictive mean $\bs{\psi_N}$; second, in one case included in Theorem
\ref{CLT} (see Section~\ref{sec-CLT} for more details), it seems to us
not immediate to check the convergence of the predictive means and so
we develop another technique that does not ask for this convergence
(see Section~\ref{proof-CLT}). Hence, the contribution of Theorem
\ref{CLT} to the SA literature is that, for a dynamics of the type
\eqref{dinamica-psin-bis} with $(\epsilon_n)_n$ and $(\delta_n)_n$
going to zero with different rates, it provides the asymptotic
behaviour of the empirical mean $\bs{\overline{\xi}_N}$, covering a
case when $\sum_n\epsilon_n=+\infty$ and $\sum_n\delta_n^2=+\infty$
and without requiring the convergence of the empirical means
$\bs{\psi_N}$.  \\ \indent Finally, it is worthwhile to point out that
we also analyze the case when $\sum_n \epsilon_n<+\infty$, which is
also excluded in SA literature and so it could be relevant in that
field. Specifically, we prove almost sure convergence of the
predictive means and of the empirical means toward a random variable
and we give a central limit theorem in the sense of stable
convergence. However, even if interesting from a theoretical point of
view, we collect these results in Section~\ref{sec-compl}, because
they are not related to the chi-squared test of goodness of fit.
\\ \indent The following statistical application of the GRP urn was
inspired by \cite{ale-cri-RP,BERTONI2018,AM2019}. However, those
papers only deal with the case when the statistics
\eqref{eq:chi1_start} is asymptotically distributed as
$\chi^2(k-1)\lambda$, with $\lambda>1$, while we also face the case
when the statistics \eqref{eq:chi1_start} is asymptotically
distributed as $\chi^2(k-1)N^{1-2e}\lambda$, illustrating a suitable
estimation procedure for the fundamental parameters $\eta=1-2e$ and
$\lambda$. To the best of our knowledge, this is the first work
presenting a model that provides a theoretical framework for a such
chi-squared test of goodness of fit.

%%%%%%%%%%%%%%%%%%%%%%%%%%%%%%%%%%%%%%%%%%%%%%%%%%%%%%%

\section{Main theorem: goodness of fit result}\label{sec-main}

Given a sample $(\bs{\xi_{1}}, \ldots, \bs{\xi_{N}})$ generated by a
GRP urn, the statistics
\begin{equation*}
O_i = \#\{n=1,\dots,N \colon \xi_{n\,i}=1\}=\sum_{n=1}^N \xi_{n\,i}, 
\, \qquad i = 1, \ldots, k,
\end{equation*}
counts the number of times we observed the value $i$.  The theorem
below states, under suitable assumptions, the almost sure convergence
of the empirical mean $\widehat{p}_i=O_i/N=\sum_{n=1}^N\xi_{n\, i}/N$
toward the probability $p_{0\,i}$, together with a chi-squared
goodness of fit test for the long-term probabilities
$p_{0\,1},\dots,p_{0\,k}$. More precisely, we prove the following
result:

\begin{theorem}\label{th-chi-squared-test}
Assume $p_{0\,i}>0$ for all $i=1,\dots,k$ and suppose to be in one of
the following cases:
\begin{itemize}
\item[a)] $\epsilon_n=(n+1)^{-\epsilon}$ and
  $\delta_n=c\epsilon_n$, with $\epsilon\in (0,1]$ and $c>0$, or 
 \item[b)] $\epsilon_n=(n+1)^{-\epsilon}$, $\delta_n\sim c(n+1)^{-\delta}$, with
   $\epsilon\in (0,1)$, $\delta\in (\epsilon/2,\epsilon)$ and $c>0$.
\end{itemize}
Define the constants $e$ and $\lambda$ as
\begin{equation*}%\label{def-e}
e=\begin{cases}
   1/2\quad &\mbox{in case a) }\\
   1/2-(\epsilon-\delta)<1/2 \quad&\mbox{in case b)}
  \end{cases}
 \end{equation*}
and
\begin{equation}\label{def-lambda}
\lambda=\begin{cases}
(c+1)^2  \quad &\mbox{in case a) with } \epsilon\in (0,1)\,,\\
(c+1)^2+c^2=[2c(c+1)+1] \quad &\mbox{in case a) with } \epsilon=1\,,\\
     \frac{c^2}{1+2(\epsilon-\delta)}\quad &\mbox{in case b)}\,.
        \end{cases}
\end{equation}
Then $\widehat{p}_i=O_i/N\stackrel{a.s.}\longrightarrow p_{0\,i}$ and 
\[
\frac{1}{N^{1-2e}}\sum_{i=1}^{k} \frac{(O_{i} - N p_{0\,i})^2}{Np_{0\,i}}=
N^{2e} \sum_{i=1}^k \frac{( \widehat{p}_i - p_{0\,i})^2}{p_{0\,i}}
\mathop{\longrightarrow}^{d}_{N\to\infty}
W_{*}= \lambda W_{0}
\]
where $W_{0}$ has distribution
$\chi^2(k-1)=\Gamma\big(\frac{k-1}{2},\frac{1}{2})$ and,
consequently, $W_{*}$ has distribution $\Gamma\big(\frac{k-1}{2},
\frac{1}{2\lambda}\big)$.
\end{theorem}

We note that $\lambda$ is a constant greater than $1$ in case a);
while, in case b), it is a strictly positive quantity. Moreover, in
case b), we have $0<(\epsilon-\delta)<\epsilon/2<1/2$ and so
$(1-2e)=2(\epsilon-\delta)\in (0,1)$.  As a consequence, we have
$N^{1-2e}\lambda>1$ for $N$ large enough.  \\

\indent In the next two examples we show that it is possible to
construct suitable sequences $(\alpha_n)_n$ and $(\beta_n)_n$ of the
model such that the corresponding sequences $(\epsilon_n)_n$ and
$(\delta_n)$ converge to zero with the same rate or with different
rates and satisfy the assumptions a) or b) of the above theorem,
respectively.

\begin{example}\label{EXA-stesso-ordine}
  {\em (Case $\epsilon_n=(n+1)^{-\epsilon}$ and
    $\delta_n=c\epsilon_n$, with $\epsilon>0$ and $c>0$ )} \\ \rm Take
  $\alpha_{n+1}= c|\bs{b_0}|(1-\beta_n)$, with $\beta_n\in [0,1)$ and
    $c>0$, that implies $ \delta_n =\frac{\alpha_{n+1}}{r_{n+1}^*} =
    c\frac{|\bs{b_0}|(1-\beta_n)}{r_{n+1}^*} = c\epsilon_n $. Set
    $r^*_n = (1+c) |\bs{b_0}|(1-t_n)$ so that from
    \eqref{dinamica-rnstar2} we obtain $t_{n+1} =\beta_n t_n $. Hence,
    we have
$$t_{n+1} = t_0 \prod_{k=0}^n \beta_k = 
\frac{c|\bs{b_0}| - |\bs{B_0}|}{(1+c)|\bs{b_0}|} \prod_{k=0}^n \beta_k
$$
and so
$$
r^*_{n+1} = 
(1+c)|\bs{b_0}|  + \big(|\bs{B_0}|- c|\bs{b_0}|\big) \prod_{k=0}^n \beta_k.
$$ Therefore, setting $\beta^*= \prod_{k=0}^\infty \beta_k \in [0,1)$,
  we get $r^*_n \longrightarrow r^* = (1+c) |\bs{b_0}| + (|\bs{B_0}| -
  c|\bs{b_0}|)\beta^* > 0$. If we choose $|\bs{B_0}|=c|\bs{b_0}|$,
  then $r_n^*=r^*=(1+c)|\bs{b_0}|$ for each $n$ and so, setting
  $\beta_n=1-(1+c)(1+n)^{-\epsilon}$ with $\epsilon>0$, we obtain
  $\epsilon_n=(1+n)^{-\epsilon}$ and $\delta_n=c\epsilon_n$. Taking
  $\epsilon\in (0,1]$, we have that $\epsilon_n$ and $\delta_n$
satisfy assumption a) of Theorem \ref{th-chi-squared-test}. Moreover,
we have $\alpha_n=c|\bs{b_0}|(1+c)n^{-\epsilon}$ and
$1-\beta_n=(1+c)(1+n)^{-\epsilon}$ and so, for the behaviour of the
factor $f(h,n)=\alpha_h\prod_{j=h}^{n-1}\beta_j$ in \eqref{eq-psi_n},
we refer to Section \ref{conti}.
\end{example}

\begin{example}\label{example-2} {\em (Case $\epsilon_n=(n+1)^{-\epsilon}$ and 
$\delta_n\sim c (n+1)^{-\delta}$, with $0<\delta<\epsilon<1$ and $c>0$)}
\\ \rm 
Take $0<\delta<\epsilon<1$ and set $\gamma = \epsilon-\delta>0$,
$r_n^* = n^\gamma$ and $(1-\beta_n)=|\bs{b_0}|^{-1}(1+n)^{-\delta}$.
We immediately have
\begin{equation*}
  \epsilon_n = |\bs{b_0}|\frac{(1-\beta_n)}{r_{n+1}^* } =
  (1+n)^{-\delta-\gamma} = (n+1)^{-\epsilon}
  \end{equation*}
and \eqref{dinamica-rnstar2} yields
$\alpha_{n+1}  = (n+1)^\gamma - n^\gamma [1-|\bs{b_0}|^{-1}(1+n)^{-\delta}] -
(1+n)^{-\delta}$, 
so that
\begin{equation*}
  \begin{split}
\delta_n  & = \frac{\alpha_{n+1}}{r_{n+1}^*}=
\frac{\alpha_{n+1}}{(n+1)^\gamma }   =
1 - \Big( 1 - \frac{1}{n+1}\Big)^\gamma 
\big[1-|\bs{b_0}|^{-1}(1+n)^{-\delta} \big] - (1+n)^{-\delta-\gamma}
\\
& = 
1 - \Big( 1 - \gamma (n+1)^{-1} + O(n^{-2})\Big)
\big[1-|\bs{b_0}|^{-1}(1+n)^{-\delta} \big] - (1+n)^{-\epsilon}
\\
& = |\bs{b_0}|^{-1}(1+n)^{-\delta} 
\Big(
1 + \gamma |\bs{b_0}| (n+1)^{-1+\delta} 
-  |\bs{b_0}| (1+n)^{-\epsilon+ \delta} 
- \gamma (n+1)^{-1}
+ O(n^{-2+\delta})
\Big)\,.
  \end{split}
  \end{equation*}
Setting $c=|\bs{b_0}|^{-1}>0$, we obtain
$\epsilon_n=(n+1)^{-\epsilon}$ and $\delta_n\sim
c(n+1)^{-\delta}$. Taking $\delta\in (\epsilon/2,\epsilon)$, we have
that $\epsilon_n$ and $\delta_n$ satisfy assumption b) of Theorem
\ref{th-chi-squared-test}. Moreover, we have $\alpha_n=
cn^{-(2\delta-\epsilon)} (1 + \gamma c^{-1} n^{-1+\delta} - c^{-1}
n^{-\epsilon+ \delta} - \gamma n^{-1} + O(n^{-2+\delta})) $ and
$(1-\beta_n)=c(1+n)^{-\delta}$, with
$0<2\delta-\epsilon<\delta<(1+2\delta-\epsilon)/2$, and so, for the
behaviour of the factor $f(h,n)=\alpha_h\prod_{j=h}^{n-1}\beta_j$ in
\eqref{eq-psi_n}, we refer to Section \ref{conti}.
\end{example}

%%%%%%%%%%%%%%%%%%%%%%%%%%%%%%%%%%%%%%%%%%%%%%%%

\section{Statistical applications}
\label{sec-appl}
In a big sample the units typically can not be assumed independent and
identically distributed, but they exhibit a structure in clusters,
with independence between clusters and with correlation inside each
cluster \cite{BERTONI2018,chessa,ieva,AM2019,tharwat,xu}. The
model and the related results presented in \cite{ale-cri-RP} and in
the present paper may be useful in the situation when inside each
cluster the probability that a certain unit chooses the value $i$ is
affected by the number of units in the same cluster that have already
chosen the value $i$, hence according to a reinforcement rule.
Formally, given a ``big'' sample $\{\bs{\xi}_n:\,n=1,\dots,N\}$, we
suppose that the $N$ units are ordered so that we have the following
$L$ clusters of units:
\begin{equation*}
C_{\ell}=\left\{\sum_{l=1}^{\ell-1}N_l+1,\dots, \sum_{l=1}^\ell N_l\right\},
\qquad \ell=1,\dots, L.
\end{equation*}
Therefore, the cardinality of each cluster $C_\ell$ is $N_\ell$. We
assume that the units in different clusters are independent, that is
$$
[\bs{\xi_1},\dots, \bs{\xi_{N_1}}],\,\dots\,,
[\bs{\xi_{\sum_{l=1}^{\ell-1}N_l+1}},\dots, \bs{\xi_{\sum_{l=1}^\ell N_l}}],\,\dots\,,
[\bs{\xi_{\sum_{l=1}^{L-1}N_l+1}},\dots, \bs{\xi_{N}}]
$$ are $L$ independent multidimensional random variables. Moreover, we
assume that the observations inside each cluster can be modelled as a
GRP satisfying case a) or case b) of Theorem
\ref{th-chi-squared-test}. Given certain (strictly positive) intrinsic
probabilities $p_{0\,1}^*(\ell),\dots, p_{0\,k}^*(\ell)$ for each
cluster $C_\ell$, we firstly want to estimate the model parameters and
then perform a test with null hypothesis
$$
H_0:\quad p_{0\,i}(\ell)=p_{0\,i}^*(\ell)\quad\forall i=1,\dots,k
$$ 
based on the the statistics
\begin{equation}\label{def-Q_ell}
Q_\ell= \frac{1}{N_\ell^{2(\epsilon-\delta)}} \sum_{i=1}^{k}
\frac{\big(O_{i}(\ell) - N_\ell p_{0\,i}^*(\ell)\big)^2}{N_\ell
  p_{0\,i}^*(\ell)},\quad\mbox{with}\; O_i(\ell)=\#\{n\in C_\ell:\,
\xi_{n\,i}=1\},
\end{equation}
and its corresponding asymptotic distribution
$\Gamma\big(\frac{k-1}{2}, \frac{1}{2\lambda}\big)$, where $\lambda$
is given in \eqref{def-lambda}. Note that we can perform the above
test for a certain cluster $\ell$, or we can consider all the clusters
together using the aggregate statistics $\sum_{\ell=1}^L Q_\ell$ and
its corresponding distribution
$\Gamma(\frac{L(k-1)}{2},\frac{1}{2\lambda})$.  \\ \indent Regarding
the probabilities $p_{0\,i}^*(\ell)$, some possibilities are:
\begin{itemize}
 \item we can take $p_{0\,i}^*(\ell)=1/k$ for all $i=1,\dots,k$ if we
   want to test possible differences in the probabilities for the $k$
   different values;
 \item we can suppose to have two different periods of times, and so
   two samples, say $\{\bs{\xi^{(1)}_n}:\,n=1,\dots,N\}$ and
   $\{\bs{\xi^{(2)}_n}:\,n=1,\dots,N\}$, take
   $p_{0\,i}^*(\ell)=\sum_{n\in C_\ell}\xi^{(1)}_{n\,i}/N_\ell$ for
   all $i=1,\dots, k$, and perform the test on the second sample in
   order to check possible changes in the intrinsic long-run
   probabilities;
  \item we can take one of the clusters as benchmark, say $\ell^*$,
    set $p_{0\,i}^*(\ell)=\sum_{n\in C_{\ell^*}}\xi_{n\,i}/N_{\ell^*}$
    for all $i=1,\dots, k$ and $\ell\neq\ell^*$, and perform the test
    for the other $L-1$ clusters in order to check differences with
    the benchmark cluster $\ell^*$.
\end{itemize}
Finally, if we want to test possible differences in the clusters, then
we can take $p_{0\,i}^*(\ell)=p_{0\,i}^*=\sum_{n=1}^N\xi_{n\,i}/N$ for
all $\ell=1,\dots, L$ and perform the test using the aggregate
statistics $\sum_{\ell=1}^L Q_\ell$ with asymptotic distribution
$\Gamma(\frac{(L-1)(k-1)}{2},\frac{1}{2\lambda})$.

%%%%%%%%%%%%%%%%%%%%%%%%

\subsection{Estimation of the parameters}
The model parameters are $\epsilon, \delta$ and $c$. However, as we
have seen, the fundamental quantities are $\eta=2(\epsilon-\delta)$
and $\lambda$ given in \eqref{def-lambda}. Moreover, recall that in
case a), we have $\eta=0$ and $\lambda>1$ and, in case b), we have
$\eta\in (0,1)$ and $\lambda>0$. Therefore, according the considered
model, the pair $(\eta,\, \lambda)$ belongs to
$S=\{0\}\times(1,+\infty)\cup (0,1)\times (0,+\infty)$.  In order to
estimate the pair $(\eta,\lambda)\in S$, we define
\begin{equation*}
  T_\ell = N_\ell^{\eta} Q_\ell=
\sum_{i=1}^{k} 
\frac{
\big(O_{i}(\ell) - N_\ell p_{0\,i}^*(\ell)\big)^2
}
{
N_\ell p_{0\,i}^*(\ell)
}\,.
\end{equation*}
Given the observed values $t_1,\dots, t_L$, the log-likelihood
function of $Q_\ell$ reads
\begin{equation*}
\ln(\mathcal{L} (\eta,\, \lambda))=\ln \mathcal{L} (\eta,\,
\lambda;\,t_1,\dots,t_L) = - \frac{k-1}{2} L\ln( \lambda )
-\frac{k-1}{2} \eta \sum_{\ell=1}^L \ln(N_\ell) -
\tfrac{1}{2\lambda}\sum_{\ell=1}^L {\tfrac{t_\ell}{N_\ell^{\eta}}}
+R_1\,,
\end{equation*}
where $R_1$ is a remainder term that does not depend on
$(\eta,\,\lambda)$.  Now, we look for the maximum likelihood estimator
of the two parameters $(\eta,\,\lambda)$.\\ \indent We immediately
observe that, when all the clusters have the same cardinality, that is
all the $N_\ell$ are equal to a certain $N_0$, then we cannot hope to
estimate $\eta$ and $\lambda$, separately. Indeed, the log-likelihood
function becomes
\begin{equation*}
\ln(\mathcal{L} (\eta,\, \lambda))=\ln \mathcal{L} (\eta,\,
\lambda;\,t_1,\dots,t_L) = - \frac{k-1}{2} L\Big[ 
\ln( \lambda )
+ \eta \ln(N_0) \Big]  -
\tfrac{1}{2\lambda N_0^\eta }\sum_{\ell=1}^L  t_\ell 
+R_1 = f (\lambda N_0^\eta )\,.
\end{equation*}
This fact implies that it possible to estimate only the parameter
$(\lambda N_0^\eta)$ as $\widehat{\lambda N_0^\eta} = \sum_{\ell=1}^L
t_\ell / (k-1)L$.\\ \indent From now on, we assume that at least two
clusters have different cardinality, that is at least a pair of
cardinalities $N_\ell$ are different.  We have to find (if they
exist!)  the maximum points of the function
$(\eta,\,\lambda)\mapsto\ln(\mathcal{L}(\eta,\, \lambda))$ on the set
$S$, which is not closed nor limited. First of all, we note that $
\ln(\mathcal{L}(\eta,\, \lambda))\to -\infty$ for $\lambda\to +\infty$
and $\lambda\to 0$.  Thus, the log-likelihood function has maximum
value on the closure $\overline{S}$ of $S$ and its maximum points are
stationary points belonging to $(0,1)\times (0,+\infty)$ or they
belong to $\{0,1\}\times (0,+\infty)$. For detecting the points of the
first type, we compute the gradient of the log-likelihood function,
obtaining
\[
\nabla_{(\eta,\,\lambda)} \ln\mathcal{L} =
\begin{pmatrix}
  - \frac{k-1}{2}  \sum_{\ell=1}^L \ln(N_\ell) 
+ \tfrac{1}{2\lambda}\sum_{\ell=1}^L {\tfrac{t_\ell\ln(N_\ell)}{N_\ell^{\eta}}} 
  \\
- \frac{k-1}{2\lambda} L 
+ \tfrac{1}{2\lambda^2}\sum_{\ell=1}^L {\tfrac{t_\ell}{N_\ell^{\eta}}}  
\end{pmatrix}\,.
\end{equation*}
Hence, the stationary points $(\eta,\,\lambda)$ of the log-likelihood
function are solutions of the system
\begin{equation*}
\left\{
\begin{aligned}
  &
  \frac{ \sum_{\ell=1}^L \tfrac{t_\ell }{N_\ell^{\eta}}\ln(N_\ell) }
       { \sum_{\ell=1}^L  \tfrac{t_\ell}{N_\ell^{\eta}} }
 = 
 \frac{ \sum_{\ell=1}^L \ln (N_\ell) }{L}
  \\
  &
  \lambda =
\frac{ \sum_{\ell=1}^L \tfrac{t_\ell}{N_\ell^{\eta}} }
     { L (k-1) }\,.
\end{aligned}
\right. 
\end{equation*}
In particular, we get that the stationary points are of the form
$(\eta,\,\lambda(\eta))$, with
\begin{equation}\label{lambda-di-eta}
\lambda(\eta) =\frac{\sum_{\ell=1}^L\tfrac{t_\ell}{N_\ell^{\eta}}}{L (k-1)}\,.
\end{equation}
\indent In order to find the maximum points on the border, that is
belonging to $\{0,1\}\times (0,+\infty)$, we observe that, fixed any
$\eta$, the function
\begin{equation*}
\lambda \mapsto - \frac{k-1}{2} L\ln( \lambda )
-
\tfrac{1}{2\lambda}\sum_{\ell=1}^L {\tfrac{t_\ell}{N_\ell^{\eta}}}
+R_2\,,
\end{equation*}
where $R_2$ is a remainder term not depending on $\lambda$, takes its
maximum value at the point $\lambda(\eta)$ defined in
\eqref{lambda-di-eta}.  \\ \indent Summing up, the problem of
detecting the maximum points of the log-likelihood function on
$\overline{S}$ reduces to the study of the maximum points on $[0,1]$
of the function
\begin{equation}\label{function}
\eta\mapsto\ln(\mathcal{L} (\eta,\, \lambda(\eta)))=
- \frac{k-1}{2} L\ln \Big(  \sum_{\ell=1}^L {\frac{t_\ell}{N_\ell^{ \eta }}}  \Big)
- \frac{k-1}{2} \eta \sum_{\ell=1}^L \ln(N_\ell)
+R_3\,,
\end{equation}
where $R_3$ is a remainder term not depending on $\eta$. To this
purpose, we note that we have 
\begin{equation*}
d \frac{\ln(\mathcal{L} (\eta,\, \lambda(\eta))}{d\eta}= 
\frac{k-1}{2} L \left[ \frac{
\sum_{\ell=1}^L {\tfrac{t_\ell }{N_\ell^{ \eta }}} \ln(N_\ell)
}{
\sum_{\ell=1}^L 
 {\frac{t_\ell}{N_\ell^{ \eta }}} 
} 
- \frac{ \sum_{\ell=1}^L \ln(N_\ell) }{ L }
\right] = \frac{(k-1)L}{2} g(\eta)\,,
\end{equation*}
where
\begin{equation*}
g(x) = 
\frac{\sum_{\ell=1}^L {\tfrac{t_\ell }{N_\ell^{x}}} \ln(N_\ell)}
     {\sum_{\ell=1}^L {\tfrac{t_\ell}{N_\ell^{x}}} }
     -  \frac{\sum_{\ell=1}^L \ln(N_\ell) }{L}\,.
\end{equation*}
Setting
\begin{equation*}
p(x,\ell)= 
\frac{\tfrac{t_\ell }{N_\ell^{x}}}
     {\sum_{l=1}^L \tfrac{t_l}{N_l^{x}}}
\end{equation*}
and denoting by $\mathop{E_x}[\cdot]$ and by $E_u[\cdot]$ the mean
value with respect to the discrete probability distribution
$\{p(x,\ell):\,\ell=1,\dots,L\}$ on $\{N_1,\dots,N_L\}$ and with
respect to the uniform discrete distribution on $\{N_1,\dots,N_L\}$
respectively, the above function $g$ can be written as
\begin{equation*}
g(x)= \sum_{\ell=1}^L p(x,\ell)\ln(N_\ell)- \frac{\sum_{\ell=1}^L
  \ln(N_\ell)}{L} = \mathop{E_x}[\ln(N)] - \mathop{E_u}[\ln(N)]\,.
\end{equation*}
Moreover, we have
\begin{equation*}
  \begin{split}
g'(x) &= 
\frac{
\Big(-\sum_{\ell=1}^L 
 {\tfrac{t_\ell }{N_\ell^{x}}} 
 \ln^2( N_\ell) \Big)
\Big(\sum_{\ell=1}^L 
 {\tfrac{t_\ell}{N_\ell^{x}}} \Big)
 +
\Big(\sum_{\ell=1}^L 
 {\tfrac{t_\ell }{N_\ell^{x}}} 
 \ln(N_\ell) \Big)^2
 }{
\Big(\sum_{\ell=1}^L 
 {\tfrac{t_\ell}{N_\ell^{x}}} \Big)^2 
}
 \\
&=
-\sum_{\ell=1}^L 
p({x},\ell)
\ln^2 (N_\ell)
+
\Big( \sum_{\ell=1}^L 
p({x},\ell)
\ln( N_\ell)
 \Big)^2 
 = - \mathop{Var_x}[\ln(N)]\,,
  \end{split}
  \end{equation*}
where $Var_{x}[\cdot]$ denotes the variance with respect to the
discrete probability distribution $\{p(x,\ell):\,\ell=1,\dots,L\}$ on
$\{N_1,\dots,N_L\}$. Since, we are assuming that at least two $N_\ell$
are different, we have $\mathop{Var_x}[\ln(N)]>0$ and so the function
$g$ is strictly decreasing. Finally, we observe that we have
\[
\mathop{Cov_u}(\ln(N),T) = 
\frac{\sum_{\ell=1}^L t_\ell \ln(N_\ell)}{L  }
 -
 \frac{\sum_{\ell=1}^L t_\ell }{L}
 \frac{ \sum_{\ell=1}^L \ln(N_\ell) }{L}
 = g(0)  \frac{\sum_{\ell=1}^L t_\ell }{L}
 \]
 and 
\[
\mathop{Cov_u}(\ln(N),\tfrac{T}{N}) = 
\frac{\sum_{\ell=1}^L \tfrac{t_\ell}{N_\ell} \ln(N_\ell)}{L  }
 -
 \frac{\sum_{\ell=1}^L \tfrac{t_\ell}{N_\ell} }{L}
 \frac{ \sum_{\ell=1}^L \ln(N_\ell) }{L}
 = g(1)  \frac{\sum_{\ell=1}^L \tfrac{t_\ell}{N_\ell} }{L}\,,
 \]
 where $Cov_u(\cdot,\cdot)$ denotes the covariance with respect to the
 discrete joint distribution concentrated on the diagonal and such
 that $P\{N=N_\ell,\,T=t_\ell\}=1/L$ with $\ell=1,\dots, L$. Hence, we
 distinguish the following cases.

\subsubsection*{First case: $\mathop{Cov_u}(\ln(N),T) \leq 0$}
We are in the case when $g(0)\leq 0$ and so the function
\eqref{function} is strictly decreasing for $\eta>0$. Thus, its
maximum value on $[0,1]$ is assumed at
$\widehat{\eta}=0$. Consequently, we have
$\widehat{\lambda}=\lambda(0) = \frac{\sum_{\ell=1}^L t_\ell
}{L(k-1)}$. Recall that we need $(0,\widehat{\lambda})\in S$ and so
$\widehat{\lambda}>1$. If the model fits well the data, this is a
consequence. Indeed, $\widehat{\lambda}$ is an unbiased estimator:
$\widehat{\lambda}\stackrel{d}\sim \Gamma(L(k-1)/2,1/(2\lambda))$ and
so $E[\widehat{\lambda}]=\lambda>1$. A value $\widehat{\lambda}\leq 1$
means a bad fit of the consider model to the data (the smaller the
value of $\lambda$, the worse the fitting). Note that in the threshold
case $(\widehat{\eta}=0,\,\widehat{\lambda}=1)$, the corresponding
test statistics \eqref{def-Q_ell} and its distribution
coincide with the classical ones used for independent observations.

\subsubsection*{Second case: $\mathop{Cov_u}(\ln(N),T) >0$ and
  $\mathop{Cov_u}(\ln(N),\tfrac{T}{N}) < 0$} We are in the case when
$g(0)>0$ and $g(1) < 0$. Hence, the function \eqref{function} has a
unique stationary point $\widehat{\eta}\in (0,1)$, which is the
maximum point. Consequently, we have $\widehat{\lambda} =
\lambda(\widehat{\eta})=\frac{\sum_{\ell=1}^L
  \tfrac{t_\ell}{N_\ell^{\widehat{\eta}}} }{L(k-1)}>0$. The point
$(\widehat{\eta},\widehat{\lambda})$ belongs to $S$.

\subsubsection*{Third case: $\mathop{Cov_u}(\ln(N),\tfrac{T}{N}) \geq 0$}
We are in the case when $g(1) \geq 0$ and so the function
\eqref{function} is strictly increasing on $[0,1]$. Hence, its maximum
point is at $\widehat{\eta}=1$, and, accordingly, we have
$\widehat{\lambda} = \lambda(1)=\frac{\sum_{\ell=1}^L
  \tfrac{t_\ell}{N_\ell} }{L(k-1)}$.  However, the point $(1,
\widehat{\lambda})$ does not belong to $S$ and so, in this case, we
conclude that we have a bad fit of the model to the data. Note that,
if the considered model fits well the data, then we have
$T/N\stackrel{d}\sim \lambda e^{ (\eta-1)\ln(N) }\chi^2(k-1)$ with
$\eta<1$ and, consequently, we expect
$\mathop{Cov_u}(\ln(N),\tfrac{T}{N})<0$. Moreover, a value $\eta\geq
1$ in the statistics \eqref{def-Q_ell} means a central limit theorem
of the type
$N^{(1-\eta)/2}(\bs{\overline{\xi}_{N}}-\bs{p_0})\stackrel{d}\sim
\mathcal{N}(0,C\Gamma)$ with $(1-\eta)/2\leq 0$.  This is impossible
since $(\bs{\overline{\xi}_{N}}-\bs{p_0})$ is bounded.

%%%%%%%%%%%%%%%%%%%%%%%%%%%%%

\section{COVID-19 epidemic Twitter analysis}
We illustrate the application of the above statistical methodology to
a data set containing posts on the on-line social network Twitter
about the COVID-19 epidemic. More precisely, the data set covers the
period from February 20th (h.~11pm) to April to 20th (h.~10pm) 2020,
including tweets in Italian language. More details on the keywords
used for the query can be found in \cite{cal-et-al}. For every
message, the relative sentiment has been calculated using the
\emph{polyglot} python module developed
in~\cite{chen2014building}. This module provides a numerical value $v$
for the sentiment and we have fixed a threshold $T=0.35$ so that we
have classified as a tweet with positive sentiment those with $v>T$
and as a tweet with negative sentiment those with $v <-T$. We have
discarded tweets with a value $v\in [-T,T]$. \\ \indent We are in the
case $k=2$ and the random variables $\xi_{n}=\xi_{n\,1}$ take the
value $1$ when the sentiment of the post $n$ is positive and the value
$0$ when the sentiment of the post $n$ is negative.  We have
partitioned the data so that each set $P_d$ collect the messages of
the single day $d$, for $d=1 \mbox{(February 20st)},\dots, 61
\mbox{(April 20th)}$ and then, in order to obtain independent
clusters, we have set $C_\ell=P_{1+3(\ell-1)}$, for $\ell=1,\dots,21 =
L$. (We have tested the independence of the timed sequence $\{Q_\ell:
\ell=1,\dots,21\}$ with a Ljung–Box test and we give the results in
Table~\ref{tab:LBtest}.) Therefore $N_\ell$ is the total number of
tweets posted during the day $1+3(\ell-1)$ and
$N=\sum_{\ell=1}^LN_\ell=699\,450$ is the sample size. \\ \indent It
is plausible that inside each cluster the sentiment associated to each
message is driven by a reinforcement mechanism, that can be modelled
by means of a GRP: the probability to have a tweet with positive
sentiment is increasing with the number of past tweets with positive
sentiment and the reinforcement is mostly driven by the most recent
tweets, in the sense explained in Section \ref{urn-model}. Note that
the main effect of the GRP urn model is the presence of ``local
fashions'', resulting in unexpected excursions of $\bs{\psi_n}$ around
the long-run probabilities $\bs{p_0}$.  In order to point out that the
considered data set exhibits this characteristics, for each $\ell$, we
have computed the daily sentiment rate $\widehat{p}_0(\ell)$, then,
according to this probability, we have generated an independent
sequence $(\xi'_{n})$ of bernoulli variables, finally we have used the
same smoothing procedure (i.e.~classical cubic spline given in R
package) to get an estimate of $\psi_{n}=\psi_{n\,1}$, for both the
real and the simulated independent data. In Fig.~\ref{fig-dependence}
the daily curves clearly show different behaviors in the two cases,
highlighting a local reinforcement among tweets.\\
\begin{figure}[tbp]
\begin{center}
\fbox{\includegraphics[width= 0.65\textwidth, height =
    0.3\textwidth]{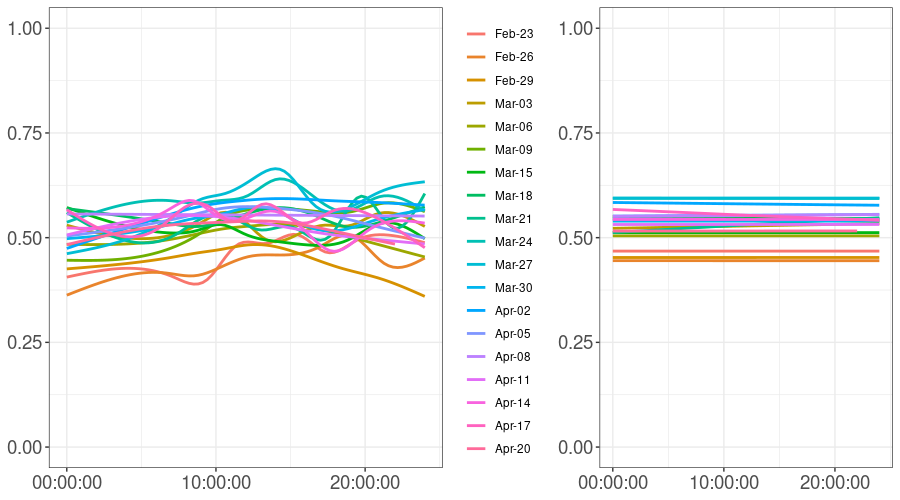}}
\end{center}
\caption{Smoothed daily estimate of $\psi_{n\,1}$ for the Twitter
  dataset (left) and for the simulated independent data (right). The
  daily mean rate $\widehat{p}_0(\ell)$ is the same for both the left
  and the right panel. $x$-axis: daily time. $y$-axis: cubic spline
  smoothing of the observed data $\xi_n$ and of the simulated independent 
  data $\xi'_n$.}
\label{fig-dependence}
\end{figure}
\begin{figure}[tbp]
\begin{center}
\fbox{\includegraphics[width= 0.65\textwidth, height =
    0.3\textwidth]{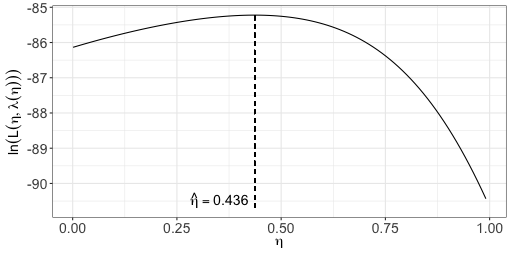}}
\end{center}
\caption{Plot of the function \eqref{function}. Its maximum point
  gives the estimated value of the model parameter $\eta$.}
\label{function-plot}
\end{figure}
\indent Our purpose is to test the null hypothesis $H_0:
\bs{p_{0}}(\ell)=\bs{p_{0}}$ for any $\ell$.  Therefore, taking
$p_{0\,1}^*(\ell)=p_0^*=\sum_{n=1}^N\xi_n/N$ for each $\ell$, we have
firstly estimated the model parameters and then we have performed the
chi-squared test based on the aggregate statistics $\sum_{\ell=1}^L
Q_\ell$ and its corresponding asymptotic distribution
$\Gamma(\frac{(L-1)(k-1)}{2},\frac{1}{2\lambda})$.  The estimated
values are $\widehat{\eta} = 0.4363572$ and $\widehat{\lambda} =
2.728098$ (in Fig.~\ref{function-plot} we plot the function
\eqref{function}).
%
\input{Table_1.tex}

The contingency table and the associated statistics for testing $H_0$
is given in Table~\ref{table-contingency}.  The obtained
$\chi^2$-statistics for a usual $\chi^2$-test is $5507.803$, which is
significant at any level of confidence. Under the proposed GRP model
and the null hypothesis, the aggregate statistics $\sum_{\ell=1}^L
Q_\ell $ has (asymptotic) distribution
$\Gamma(\frac{L-1}{2},\frac{1}{2\widehat{\lambda}})$ and the
corresponding $p$-value associated to the data is equal to
$0.4579297$.  The null hypothesis that the daily long-run sentiment
rate of the posts is the same for all the considered days is therefore
strongly rejected with a classical $\chi^2$ test, while the same
hypothesis is accepted if we take into account the reinforcement
mechanism of correlation given in GRP model.\\ \indent In
Fig. \ref{Ql-plot} there are the values of the single statistics
$Q_\ell$ compared to the $95th$-quantile of the distribution
$\Gamma(\frac{1}{2},\frac{1}{2\widehat{\lambda}})$.
\begin{figure}[tbp]
\begin{center}
\fbox{\includegraphics[width= 0.65\textwidth]{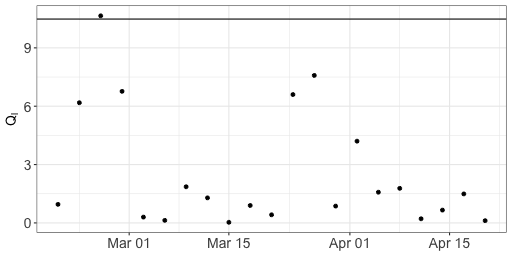}}
\end{center}
\caption{Plot of the $Q_\ell$-series. The black line corresponds to
  the value of $95th$-quantile of the distribution
  $\Gamma(\frac{1}{2},\frac{1}{2\widehat{\lambda}})$, that is
  $10.48$.}
\label{Ql-plot}
\end{figure}
%
\input{TestBox2.tex}

%%%%%%%%%%%%%%%%%%%%%%%%%%%%%%%%%%%%%%%%%%%%%%%%%%%%%%%%

\section{Asymptotic results for the empirical means}\label{sec-CLT}

Theorem \ref{th-chi-squared-test} is a consequence of the following
Proposition\ref{CLT-stesso-ordine} and Theorem~\ref{CLT} for the
empirical means $\bs{\overline{\xi}_N}$. In the sequel, we will use
the symbol $\stackrel{s}\longrightarrow$ in order to denote the stable
convergence (for a brief review on stable convergence, see
Section~\ref{stable-conv-app}).  \\

\indent Leveraging the Stochastic Approximation results collected in
Section~\ref{SA-app}, we prove in 
Section~\ref{proof-CLT-stesso-ordine} the following result:
\begin{proposition}\label{CLT-stesso-ordine}
Take $\epsilon_n=(n+1)^{-\epsilon}$ and $\delta_n=c\epsilon_n$, with
$\epsilon\in (0,1]$ and $c>0$, and set $\Gamma=\mathrm{diag}
  ({\bs{p_0}}) - {\bs{p_0}}{\bs{p_0}^\top}$.  Then
  $\bs{\overline{\xi}_N}\stackrel{a.s.}\longrightarrow \bs{p_0}$ and
\begin{equation*}
\sqrt{N}\left(\bs{\overline{\xi}_N}-\bs{p_0}\right)
\stackrel{s}\longrightarrow 
\mathcal{N}\left(\bs{0}, \lambda \Gamma \right), 
\end{equation*}
with $\lambda=(c+1)^2$ when $0<\epsilon<1$ and
$\lambda=(c+1)^2+c^2=2c(c+1)+1$ when $\epsilon=1$.
\end{proposition}

For the case when $(\epsilon_n)_n$ and $(\delta_n)_n$ in
\eqref{dinamica-psin-bis} go to zero with different rates, we prove
the following theorem (the proof is illustrated in Section
\ref{proof-CLT}):

\begin{theorem}\label{CLT}
Take $\epsilon_n=(n+1)^{-\epsilon}$ and $\delta_n\sim c(n+1)^{-\delta}$, with
$\epsilon\in (0,1)$, $\delta\in (\epsilon/2,\epsilon)$ and $c>0$. Then
$\bs{\overline{\xi}_N}\stackrel{a.s.}\longrightarrow \bs{p_0}$ and
\begin{equation*}
N^{1/2-(\epsilon-\delta)}\left(\bs{\overline{\xi}_N}-\bs{p_0}\right)
\stackrel{s}\longrightarrow 
\mathcal{N}\left(\bs{0}, \frac{c^2}{1+2(\epsilon-\delta)}\Gamma \right), 
\end{equation*}
with $\Gamma=\mathrm{diag} ({\bs{p_0}}) - {\bs{p_0}}{\bs{p_0}^\top}$. 
\end{theorem}
In the framework of the above theorem, we can distinguish the
following two cases:
\begin{itemize}
\item[1)] $\epsilon\in (1/2,1)$ and $\delta\in (1/2,\epsilon)$ or 
\item[2)] $\epsilon\in (0,1)$ and $\delta\in
  (\epsilon/2,\min\{\epsilon,1/2\}]\setminus\{\epsilon\}$.
\end{itemize}
In case 1), we have $\sum_n \epsilon_n=+\infty$,
$\sum_n\epsilon_n^2<+\infty$ and $\sum_n\delta_n^2<+\infty$ and so the
typical asymptotic behaviour of the predictive mean of an urn process,
that is its almost sure convergence. In case 2), we have
$\sum_n\epsilon_n=+\infty$ and $\sum_n\delta_n^2=+\infty$ (while the
series $\sum_n\epsilon_n^2$ may be convergent or divergent) and it
seems to us not immediate to check the convergence of the predict
means. Therefore, for the proof of Theorem \ref{CLT} in this last
case, we will employ a different technique, which is based on the
$L^2$-estimate of Lemma \ref{lemma-pre} for the predictive mean
$\bs{\psi_N}$ and the almost sure convergence of the corresponding
empirical mean $\bs{\overline{\psi}_{N-1}}$.

%%%%%%%%%%%%%%%%%%%%%%%%%%%%%%%%%%%%%%%%%%%%%%%%%%%%%%%%%%%%%

\section{Proof of Theorem \ref{CLT}}
\label{proof-CLT}

For all the sequel, we set
$\bs{\overline{\psi}_{N-1}}=\sum_{n=1}^{N}\bs{\psi_{n-1}}/N$ and
$\bs{\overline{\theta}_{N-1}}=\sum_{n=1}^{N}\bs{\theta_{n-1}}/N$.  To
the proof of Theorem \ref{CLT}, we premise some intermediate results.

\begin{lemma}\label{lemma-pre}
Under the same assumptions of Theorem \ref{CLT}, we have
$E[\|\bs{\theta_n}\|^2]=O(n^{\epsilon-2\delta})\to 0$. 
\end{lemma}
\begin{proof} We observe that, starting from \eqref{dinamica-psin-bis}, we get  
$$ \|\bs{\theta_{n+1}}\|^2=\bs{\theta_{n+1}}^{\top}\bs{\theta_{n+1}}=
  (1-\epsilon_n)^2\|\bs{\theta_n}\|^2+\delta_n^2\|\Delta\bs{M_{n+1}}\|^2
  +2(1-\epsilon_n)\delta_n\bs{\theta_n}^{\top}\Delta\bs{M_{n+1}}
$$
and so
\begin{equation}\label{eq-cond}
E[\|\bs{\theta_{n+1}}\|^2|\mathcal{F}_n]=
(1-\epsilon_n)^2 \|\bs{\theta_n}\|^2
+\delta_n^2E[\|\Delta\bs{M_{n+1}}\|^2|\mathcal{F}_n]\,.
\end{equation}
Hence, setting $x_n=E[\|\bs{\theta_{n}}\|^2]$, we get
\begin{equation*}%\label{dinamica-xn}
\begin{split}
x_{n+1}&=(1-2\epsilon_n)x_n+\epsilon_n^2x_n+\delta_n^2E[\|\Delta \bs{M_{n+1}}\|^2]
\\
&=(1-2\epsilon_n)x_n+
\epsilon_n
\left(
\epsilon_n x_n+\frac{\delta_n^2}{\epsilon_n}E[\|\Delta \bs{M_{n+1}}\|^2]
\right)\\
&=(1-2\epsilon_n)x_n+2\epsilon_n\zeta_n,
\end{split}
\end{equation*}
with $0\leq\zeta_n=\left( \epsilon_n
x_n+\frac{\delta_n^2}{\epsilon_n}E[\|\Delta \bs{M_{n+1}}\|^2]
\right)/2$.  Applying Lemma \ref{lemma-del}  
(with $\gamma_n=2\epsilon_n$), we find that $\limsup_n
x_n\leq\limsup_n\zeta_n$. On the other hand, since
$(\Delta\bs{M_{n+1}})_n$ is uniformly bounded and
$\epsilon_n^2/\delta_n^2\sim c^{-2}n^{-2(\epsilon-\delta)}\to 0$, we have
$\zeta_n=O(\epsilon_n+\delta_n^2\epsilon_n^{-1})=O(\delta_n^2/\epsilon_n)$
and so $x_n=O(\delta_n^2/\epsilon_n)$. We can conclude recalling that
$\delta_n^2/\epsilon_n\sim c^2 n^{\epsilon-2\delta}$.
\end{proof}

\begin{lemma}\label{lemma-decompo}
Under the same assumptions of Theorem \ref{CLT}, we have
\begin{equation}\label{eq-media-theta}
\bs{\overline{\theta}_{N-1}}=
\frac{1}{N}\sum_{n=1}^{N} \bs{\theta_{n-1}}=
\frac{1}{N}\sum_{n=0}^{N-1}\frac{\delta_n}{\epsilon_n}\Delta\bs{M_{n+1}}
+\bs{R_N}\,,
\end{equation}
where $\bs{R_N}\stackrel{a.s.}\longrightarrow \bs{0}$ and
$N^eE\big[\,|\bs{R_N}|\,\big]\longrightarrow 0$ with
$e=1/2-(\epsilon-\delta)\in (0,1/2)$.
\end{lemma}
\begin{proof} By \eqref{dinamica-thetan}, we have 
\begin{equation*}%\label{eq-theta}
\bs{\theta_{n}}=
-\frac{1}{\epsilon_{n}}\left(\bs{\theta_{n+1}}-\bs{\theta_{n}}\right)
+\frac{\delta_n}{\epsilon_n}\Delta\bs{ M_{n+1}}.
\end{equation*}
Therefore, we can write
\begin{equation*}
\begin{split}
\sum_{n=0}^{N-1} \bs{\theta_{n}}&=
-\sum_{n=0}^{N-1} 
\frac{1}{\epsilon_{n}}\left(\bs{\theta_{n+1}}-\bs{\theta_{n}}\right)
+\sum_{n=0}^{N-1}\frac{\delta_n}{\epsilon_n}\Delta\bs{ M_{n+1}}\\
&=
-\left(
\frac{\bs{\theta_{N}}}{\epsilon_{N-1}}-\frac{\bs{\theta_{0}}}{\epsilon_0}
\right)
-\sum_{n=1}^{N-1}
\left(\frac{1}{\epsilon_{n-1}}-\frac{1}{\epsilon_n}\right)\bs{\theta_n}
+\sum_{n=0}^{N-1}\frac{\delta_n}{\epsilon_n}\Delta\bs{ M_{n+1}}\,,
\end{split}
\end{equation*}
where the second equality is due to the Abel transformation for a series.
It follows the decomposition \eqref{eq-media-theta} with 
\begin{equation}\label{def-Rn}
\bs{R_N}=
-\frac{1}{N}\left(
\frac{\bs{\theta_{N}}}{\epsilon_{N-1}}-\frac{\bs{\theta_{0}}}{\epsilon_0}
\right)
-\frac{1}{N}\sum_{n=1}^{N-1}
\left(\frac{1}{\epsilon_{n-1}}-\frac{1}{\epsilon_n}\right)\bs{\theta_n}\,.
\end{equation}
Since $|\bs{\theta_n}|=O(1)$, we have
$$
|\bs{R_N}|= 
O(N^{-1}\epsilon_{N-1}^{-1})+
O\left(N^{-1}\sum_{n=1}^{N-1}|\epsilon_{n-1}^{-1}-\epsilon_n^{-1}|
\right)
$$ Note that $\sum_{n=1}^{N-1}|\epsilon_{n-1}^{-1}-\epsilon_n^{-1}|=
\epsilon_0^{-1}-\epsilon_{N-1}^{-1}$ when $(\epsilon_n)$ is decreasing
and so the last term in the above expression is
$O(N^{-1}\epsilon_{N-1}^{-1})$. Therefore, since $\epsilon<1$ by
assumption, we have $|\bs{R_N}|=O(N^{-(1-\epsilon)})\to 0$.
\\ \indent Regarding the last statement of the lemma, we observe that,
from what we have proven before, we obtain
$N^eE\big[|\bs{R_N}|\big]=O(N^{e-(1-\epsilon)})=O(N^{\delta-1/2})\to
0$ when $\delta<1/2$. However, in the considered cases 1) and 2), we 
might have $\delta\geq 1/2$. Therefore, we need other arguments in
order to prove the last statement. To this purpose, we observe that,
by Lemma \ref{lemma-pre}, we have $E[\,|\bs{\theta_n}|\,]=
O(n^{\epsilon/2-\delta})$ and so, by \eqref{def-Rn}, we have
\begin{equation*}
\begin{split}
N^eE\big[\,|\bs{R_N}|\,\big]&=O(N^{-(1-e)}N^{3\epsilon/2-\delta})+
O\left(\frac{1}{N^{1-e}}\sum_{n=1}^{N-1}
|\epsilon_{n-1}^{-1}-\epsilon_n^{-1}| n^{\epsilon/2-\delta}\right)
\\
&=O(N^{-(1-\epsilon)/2})+ O\left(\frac{1}{N^{1-e}}\sum_{n=1}^{N-1}
|\epsilon_{n-1}^{-1}-\epsilon_n^{-1}| n^{\epsilon/2-\delta}\right)\,.
\end{split}
\end{equation*}
Moreover, we have 
$$
\sum_{n=1}^{N-1}
  |\epsilon_{n-1}^{-1}-\epsilon_n^{-1}| n^{\epsilon/2-\delta}=
\sum_{n=1}^{N-1}\left[(n-1)^{\epsilon}-n^{\epsilon}\right] n^{\epsilon/2-\delta}=
\sum_{n=1}^{N-1} n^{\epsilon-1+\epsilon/2-\delta}
\sim N^{3\epsilon/2-\delta}
=
o(N^{1-e})\,,
$$ because $e=1/2-(\epsilon-\delta)$ and $\epsilon<1$. Summing up, we
have $N^eE[|\bs{R_N}|]=O(N^{-(1-\epsilon)/2})+o(1)\to 0$.
\end{proof}

\begin{lemma}\label{lemma-conv-medie}
Under the same assumptions of Theorem \ref{CLT}, we have
$\bs{\overline{\theta}_{N-1}}\stackrel{a.s.}\longrightarrow \bs{0}$,
that is $\bs{\overline{\psi}_{N-1}}\stackrel{a.s.}\longrightarrow
\bs{p_0}$. In particular, when $\epsilon\in (1/2,1)$ and $\delta\in
(1/2,\epsilon)$, we have
$\bs{\theta_N}\stackrel{a.s.}\longrightarrow\bs{0}$, that is
$\bs{\psi_N}\stackrel{a.s.}\longrightarrow\bs{p_0}$.
\end{lemma}
\begin{proof} Let us distinguish the following two cases:
\begin{itemize}
\item[1)] $\epsilon\in (1/2,1)$ and $\delta\in (1/2,\epsilon)$ or 
\item[2)] $\epsilon\in (0,1)$ and $\delta\in
  (\epsilon/2,\min\{\epsilon,1/2\}]\setminus\{\epsilon\}$.
\end{itemize}
For the case 1), we observe that, by \eqref{eq-cond}, we have 
$$
E[\|\bs{\theta_{n+1}}\|^2|\mathcal{F}_n]
\leq (1+\epsilon_n^2)E[\|\bs{\theta_n}\|^2|\mathcal{F}_n]
+\delta_n^2E[\|\Delta\bs{ M_{n+1}}\|^2|\mathcal{F}_n].
$$ Therefore, since $(\Delta\bs{M_{n+1}})_n$ is uniformly bounded and,
in case 1), we have $\sum_n\epsilon_n^2<+\infty$ and
$\sum_n\delta_n^2<+\infty$, the sequence $(\|\bs{\theta_n}\|^2)_n$ is
a bounded non-negative almost supermartingale. As a consequence, it
converges almost surely to a certain random variable. This limit
random variable is necessarily equal to $\bs{0}$ because, by Lemma
\ref{lemma-pre}, we have
$E[\|\bs{\theta_n}\|^2]=O(n^{\epsilon-2\delta})\to 0$. Hence, we have
the almost sure convergence of $\bs{\theta_N}$ to $\bs{0}$ and,
consequently, the almost sure convergence of
$\bs{\overline{\theta}_{N-1}}$ to $\bs{0}$ follows by Lemma
\ref{lemma-serie-rv} and Remark \ref{case-as-conv} 
(with $c_n=n$ and $v_{N,n}=n/N$), because
$E[\bs{\theta_{n-1}}|\mathcal{F}_{n-2}]=(1-\epsilon_{n-2})\bs{\theta_{n-2}}\to
\bs{0}$ almost surely.  \\ \indent For the case 2), we use Lemma
\ref{lemma-decompo}, that gives the decomposition
\eqref{eq-media-theta}, with
$\bs{R_N}\stackrel{a.s.}\longrightarrow\bs{0}$. Indeed, by this
decomposition, it is enough to prove that the term
$\sum_{n=0}^{N-1}\frac{\delta_n}{\epsilon_n}\Delta\bs{M_{n+1}}/N$
converges almost surely to $\bs{0}$. To this purpose, we observe that,
if we set
$$
\bs{L_n}=
\sum_{j=1}^{n}\frac{1}{j}\frac{\delta_{j-1}}{\epsilon_{j-1}}\Delta\bs{M_{j}},
$$ then $(\bs{L_n})$ is a square integrable martingale. Indeed, we
have
$$ \sum_{n=1}^{+\infty}\frac{1}{n^{2}}\frac{\delta_{n-1}^2}{\epsilon_{n-1}^2}
E[\|\Delta\bs{M_{n}}\|^2] =O\left(\sum_{n=1}^{+\infty}
\frac{1}{n^{1+2e}}\right)<+\infty\,.
$$ Therefore, $(\bs{L_n})$ converges almost surely, that
is we have
$\sum_{n}\frac{1}{n}\frac{\delta_{n-1}}{\epsilon_{n-1}}\Delta\bs{M_{n}}<+\infty$
almost surely.  Applying Lemma \ref{gen-kro-lemma} 
(with $v_{N,n}=n/N$), we find
$$
  \frac{1}{N}\sum_{n=0}^{N-1}\frac{\delta_n}{\epsilon_n}\Delta\bs{M_{n+1}}
  =\sum_{n=1}^N v_{N,n}
  \frac{1}{n}\frac{\delta_{n-1}}{\epsilon_{n-1}}\Delta\bs{M_{n}}
  \stackrel{a.s.}\longrightarrow \bs{0}
$$ 
and so $\bs{\overline{\theta}_{N-1}}\stackrel{a.s.}\longrightarrow \bs{0}$.  
\end{proof}

\noindent {\bf Proof of Theorem \ref{CLT}.}  Set $e=1/2-(\epsilon-\delta)\in
(0,1/2)$ and $\lambda=c^2/[2(1-e)]=c^2/[1+2(\epsilon-\delta)]$. 
Moreover, let us distinguish the following two cases:
\begin{itemize}
\item[1)] $\epsilon\in (1/2,1)$ and $\delta\in (1/2,\epsilon)$ or 
\item[2)] $\epsilon\in (0,1)$ and $\delta\in
  (\epsilon/2,\min\{\epsilon,1/2\}]\setminus\{\epsilon\}$.
\end{itemize}
\noindent {\em Almost sure convergence:} In case 1), by Lemma
\ref{lemma-conv-medie}, $\bs{\psi_N}$ converges almost surely to
$\bs{p_0}$. Therefore, the almost sure convergence of
$\bs{\overline{\xi}_N}$ to $\bs{p_0}$ follows by Lemma
\ref{lemma-serie-rv} and Remark \ref{case-as-conv} 
(with $c_n=n$ and $v_{N,n}=n/N$), because
$E[\bs{\xi_{n+1}}|\mathcal{F}_n]=\bs{\psi_n}\to \bs{p_0}$ almost
surely and $\sum_{n} E[\|\bs{\xi_n}\|^2 ]n^{-2}\leq \sum_{n}
n^{-2}<+\infty$.  \\ \indent In case 2), we use a different
argument. Take $\gamma\in [0,e)$ and set
$$
\bs{L_n}=\sum_{j=1}^{n}\frac{1}{j^{1-\gamma}}\frac{\delta_{j-1}}{\epsilon_{j-1}}
\Delta\bs{M_{j}}\,.
$$
Then $(\bs{L_n})$ is a square integrable martingale, because we have 
$$ \sum_{n=1}^{+\infty}\frac{1}{n^{2-2\gamma}}\frac{\delta_{n-1}^2}{\epsilon_{n-1}^2}
E[\|\Delta\bs{M_{n}}\|^2] =O\left(\sum_{n=1}^{+\infty}
\frac{1}{n^{1+2e-2\gamma}}\right)<+\infty\,.
$$ Therefore, $(\bs{L_n})$ converges almost surely, that is we have
$\sum_{n}\frac{1}{n^{1-\gamma}}\frac{\delta_{n-1}}{\epsilon_{n-1}}
\Delta\bs{M_{n}}<+\infty$ almost surely. By Lemma \ref{gen-kro-lemma} % in appendix 
(with $v_{N,n}=(n/N)^{1-\gamma}\epsilon_{n-1}/\delta_{n-1}\sim
n^{1-\gamma-\epsilon+\delta}/N^{1-\gamma}$), we get
$$
\frac{1}{N^{1-\gamma}}
\sum_{n=0}^{N-1}\Delta\bs{M_{n+1}}=
\sum_{n=1}^{N} v_{N,n}\frac{1}{n^{1-\gamma}}\frac{\delta_{n-1}}{\epsilon_{n-1}}
\Delta\bs{M_{n}}
\stackrel{a.s.}\longrightarrow\bs{0}.
$$ 
Therefore, we have
$$
N^{\gamma}\left(\bs{\overline{\xi}_N}-\bs{\overline{\psi}_{N-1}}\right)=
\frac{1}{N^{1-\gamma}}
\sum_{n=0}^{N-1}\Delta\bs{M_{n+1}}\stackrel{a.s.}\longrightarrow\bs{0},
$$ that is
$\left(\bs{\overline{\xi}_N}-\bs{\overline{\psi}_{N-1}}\right)=o(N^{-\gamma})$
for each $\gamma\in [0,e)$. Recalling Lemma
\ref{lemma-conv-medie}, we obtain in particular that
$\bs{\overline{\xi}_N}$ converges almost surely to $\bs{p_0}$.
\\
\noindent{\em Second order asymptotic behaviour:} We have
\begin{equation}\label{eq-finale}
N^{e}\left(\bs{\overline{\xi}_N}-\bs{p_0}\right)=
N^{e}\bs{\overline{\mu}_{N}}=
N^{e-1/2}\sqrt{N}\left(\bs{\overline{\mu}_N}-\bs{\overline{\theta}_{N-1}}\right)+
N^{e}\bs{\overline{\theta}_{N-1}}\,.
\end{equation}
Moreover, by Lemma \ref{lemma-pre}, we have
\begin{equation*}
\begin{split}
\frac{1}{N}\sum_{n=0}^{N-1}E[|\bs{\theta_{n}}|]
&=O(N^{-1}\sum_{n=1}^{N} n^{\epsilon/2-\delta})
=O(N^{-1-\delta+\epsilon/2+1})=O(N^{\epsilon/2-\delta})\to 0\,,
\\
\frac{1}{N}\sum_{n=0}^{N-1}E[\|\bs{\theta_{n}}\|^2]
&=O(N^{-1}\sum_{n=1}^{N} n^{\epsilon-2\delta})
=O(N^{-1-2\delta+\epsilon+1})=O(N^{\epsilon-2\delta})\to 0\,,
\end{split}
\end{equation*}
and so Theorem \ref{preliminary-CLT} holds true with $V=\Gamma$ (see
Remark \ref{remark-preliminary-CLT}).  Therefore, the first term in
the right side of \eqref{eq-finale} converges in probability to
$\bs{0}$ because $e<1/2$. Hence, if we prove that
\begin{equation}\label{eq-finale-2}
N^{e}\bs{\overline{\theta}_{N-1}}\stackrel{s}\longrightarrow 
\mathcal{N}\left(\bs{0}, \lambda \Gamma \right)\,,
\end{equation}
then the proof is concluded.  \\ \indent In order to prove
\eqref{eq-finale-2}, we observe that, by decomposition
\eqref{eq-media-theta} in Lemma \ref{lemma-decompo}, we have
\begin{equation*}%\label{eq-media-theta-2}
N^e\bs{\overline{\theta}_{N-1}}=\sum_{n=1}^N\bs{Y_{N,n}}+N^e\bs{R_N}\,,
\end{equation*}
where
$\bs{Y_{N,n}}=\frac{1}{N^{1-e}}\frac{\delta_{n-1}}{\epsilon_{n-1}}\Delta\bs{M_n}
$ and $N^e\bs{R_N}$ converges in probability to $\bs{0}$ (because
$N^eE\big[|\bs{R_N}|]\to 0$). Therefore, it is enough to prove that
the term $\sum_{n=1}^N\bs{Y_{N,n}}$ stably converges to the Gaussian
kernel $\mathcal{N}(0,\lambda\Gamma)$, with
$\lambda=c^2/[2(1-e)]=c^2/[1+2(\epsilon-\delta)]$. To this purpose, we
observe that $E[\bs{Y_{N,n}}|\mathcal{F}_{n-1}]=\bs{0}$ and so
$\sum_{n=1}^N\bs{Y_{N,n}}$ converges stably to
$\mathcal{N}(\bs{0},\lambda\Gamma)$ if the conditions (c1) and (c2) of
Theorem \ref{stable-conv},  
with $V=\lambda\Gamma$, hold
true.  Regarding (c1), we note that $\delta_{n-1}/\epsilon_{n-1}\sim c
n^{\epsilon-\delta}=c n^{1/2-e}$ and so we have
$$
\max_{1\leq n\leq
  N} |\bs{Y_{N,n}} | \leq N^{-(1-e)}\max_{1\leq n\leq N}
\frac{\delta_{n-1}}{\epsilon_{n-1}}|\bs{\xi_{n}} -{\bs{\psi_{n-1}}}|
\leq N^{-(1-e)}\max_{1\leq n\leq N}
\frac{\delta_{n-1}}{\epsilon_{n-1}}=O(N^{-1/2})\to 0\,.
$$ 
Condition (c2) means 
\begin{equation}\label{c2}
\frac{1}{N^{2(1-e)}}\sum_{n=1}^N
\frac{\delta_{n-1}^2}{\epsilon_{n-1}^2}
(\bs{\xi_{n}}-{\bs{\psi_{n-1}}})(\bs{\xi_{n}}-{\bs{\psi_{n-1}}})^\top
\stackrel{P}{\longrightarrow}\lambda \Gamma.
\end{equation}
We note that
$N^{-2(1-e)}\sum_{n=1}^{N}\delta_{n-1}^2/\epsilon_{n-1}^2\to \lambda$, 
because $\delta_{n-1}^2/\epsilon_{n-1}^2\sim c^2 n^{1-2e}$, and
$$
E[
(\bs{\xi_{n}}-{\bs{\psi_{n-1}}})(\bs{\xi_{n}}-{\bs{\psi_{n-1}}})^\top
|\mathcal{F}_{n-1}]=
\mathrm{diag} ({\bs{\psi_{n-1}}}) - {\bs{\psi_{n-1}}}
{\bs{\psi_{n-1}}^\top}\,.
$$ Therefore, in case 1), condition \eqref{c2} immediately follows by
the almost sure convergence of $\bs{\psi_n}$ to $\bs{p_0}$. It is
enough to apply Lemma \ref{lemma-serie-rv} and Remark
\ref{case-as-conv} 
with $c_n=n$ and
$v_{N,n}=n\delta_{n-1}^2/(N^{2(1-e)}\epsilon_{n-1}^2)\sim c^2
n^{1+2(\epsilon-\delta)}/N^{2-2e}=c^2(n/N)^{2(1-e)}$.  In case 2), we
apply again Lemma \ref{lemma-serie-rv} with the above $c_n$ and
$v_{N,n}$, but we note that $\bs{\psi_n}=\bs{\theta_n}+\bs{p_0}$ and
so condition \eqref{conv-medie-cond} in Lemma \ref{lemma-serie-rv},
with $V=\lambda\Gamma$, is equivalent to
\begin{equation*}%\label{cond-conv-medie-pesate}  
\frac{1}{N^{2-2e}}\sum_{n=0}^{N-1}\frac{\delta_n^2}{\epsilon_n^2}
\bs{\theta_{n}}\stackrel{P}\longrightarrow\bs{0}
\qquad\mbox{and}\qquad
\frac{1}{N^{2-2e}}\sum_{n=0}^{N-1}\frac{\delta_n^2}{\epsilon_n^2}
\bs{\theta_{n}}\bs{\theta_{n}}^{\top}
\stackrel{P}\longrightarrow 0_{k\times k}.
\end{equation*}
These two convergences hold true because, by Lemma \ref{lemma-pre}, we have 
\begin{equation*}
\begin{split}
\frac{1}{N^{2-2e}}\sum_{n=0}^{N-1}\frac{\delta_n^2}{\epsilon_n^2}E[|\bs{\theta_n}|]
&=O(N^{-2+2e}\sum_{n=1}^{N} n^{-2\delta+2\epsilon-\delta+\epsilon/2})
=O(N^{-2+2e-3\delta+5\epsilon/2+1})=O(N^{-\delta+\epsilon/2})\to 0\,,
\\
\frac{1}{N^{2-2e}}\sum_{n=0}^{N-1}\frac{\delta_n^2}{\epsilon_n^2}
E[\|\bs{\theta_n}\|^2]&=
O(N^{-2+2e}\sum_{n=1}^N n^{-2\delta+2\epsilon-2\delta+\epsilon})
=O(N^{-2+2e-4\delta+3\epsilon+1})=O(N^{-2\delta+\epsilon})\to 0\,.
\end{split}
\end{equation*}
Therefore, in both cases 1) and 2), conditions c1) and c2) of Theorem 
\ref{stable-conv} 
are satisfied and so
$\sum_{n=1}^N\bs{Y_{N,n}}$ stably converges to the Gaussian kernel
$\mathcal{N}(0,\lambda\Gamma)$.  \qed

%%%%%%%%%%%%%%%%%%%%%%%%%%%%%%%%%%%%%%%%%%%%%%%%%%%%%%%
%	\section*{Declaration}
%Both authors equally contributed to this work.  
%\par
%
%%%%%%%%%%%%%%%%%%%%%%%%%%%%%%%%%%%%%%%%%%%%%%%%%%%%%%%%%%%%%%%%%%%
%
%\begin{acknowledgements}
%%\section*{Acknowledgements}
% The authors thank Fabio Saracco for the dataset
%used for the applicative example.\\ \indent 
%Giacomo Aletti is a member of the Italian Group ``Gruppo
%Nazionale per il Calcolo Scientifico'' of the Italian Institute
%``Istituto Nazionale di Alta Matematica'' and Irene Crimaldi is a
%member of the Italian Group ``Gruppo Nazionale per l'Analisi
%Matematica, la Pro\-ba\-bi\-li\-t\`a e le loro Applicazioni'' of the
%Italian Institute ``Istituto Nazionale di Alta Ma\-te\-ma\-ti\-ca''.
%\end{acknowledgements}
%
%
%% Authors must disclose all relationships or interests that 
%% could have direct or potential influence or impart bias on 
%% the work: 
%%
%\section*{Conflict of interest}
%The authors declare that they have no conflict of interest.

%%%%%%%%%%%%%%%%%%%%%%%%%%%%%%%%%%%%%%%%%%%%%%%%%%%%%%
\noindent{\bf Declaration}\\

\noindent Both authors equally contributed to this work.  
\\

\noindent{\bf Acknowledgments}\\

\noindent Giacomo Aletti is a member of the Italian Group ``Gruppo
Nazionale per il Calcolo Scientifico'' of the Italian Institute
``Istituto Nazionale di Alta Matematica'' and Irene Crimaldi is a
member of the Italian Group ``Gruppo Nazionale per l'Analisi
Matematica, la Pro\-ba\-bi\-li\-t\`a e le loro Applicazioni'' of the
Italian Institute ``Istituto Nazionale di Alta Ma\-te\-ma\-ti\-ca''.
\\

\noindent{\bf Funding Sources}\\

\noindent Irene Crimaldi is partially supported by the Italian
``Programma di Attivit\`a Integrata'' (PAI), project ``TOol for
Fighting FakEs'' (TOFFE) funded by IMT School for Advanced Studies
Lucca.

\pagebreak
%%%%%%%%%%%%%%%%%%%%%%%%%%%%%%%%%%%%%%%%%%%%%%%%%%%%%%%%%%%%%%%%%%%
\appendix

\addtocontents{toc}{\vfill\eject\noindent\hrulefill\par}
\addtocontents{toc}{~\hfill\textbf{Page}\par}

\section*{SM Supplemental Materials}
\addcontentsline{toc}{section}{SM Supplemental Materials}

%%%%%%%%%% Merge with supplemental materials %%%%%%%%%%
%%%%%%%%%% Prefix a "S" to all equations, figures, tables and reset the counter %%%%%%%%%%
\setcounter{equation}{0}
\setcounter{figure}{0}
\setcounter{table}{0}
\setcounter{page}{1}
\setcounter{page}{1}
\makeatletter

\renewcommand{\refname}{SR References}
\renewcommand{\thefigure}{S\arabic{figure}}
%\renewcommand*{\bibnumfmt}[1]{[S#1]}
%\renewcommand*{\citenumfont}[1]{S#1}
%%%%%%%%%% Prefix a "S" to all equations, figures, tables and reset the counter %%%%%%%%%%

\newcommand*\mref[1]{%
    \ref{#1}}

% To add an 'M' prefix to a equation
\newcommand*\meqref[1]{%
    (\ref{#1})}

\def\thepage{\underline{S\,\arabic{page}}}
\def\theequation{S\arabic{section}.\arabic{equation}}
\def\thesection{S\arabic{section}}
\def\theremark{S\arabic{section}.\arabic{remark}}
\def\thetheorem{S\arabic{section}.\arabic{theorem}}
\def\thelemma{S\arabic{section}.\arabic{lemma}}
\def\theexample{S\arabic{section}.\arabic{example}}
\makeatletter
\renewcommand\@bibitem[1]{\item\if@filesw \immediate\write\@auxout
    {\string\bibcite{#1}{S\the\value{\@listctr}}}\fi\ignorespaces}% <------------
\def\@biblabel#1{[S#1]}% <-------------------
\makeatother
%\makeatletter
%\addtocounter{\@listctr}{48}
%\makeatother

In this document we collect some proofs, complements, technical
results and recalls, useful for \cite{S:AleCri_new20}.  
Therefore, the notation and the assumptions used here are the same as those used in that paper.

\section{Proofs and intermediate results}\label{sec-proofs}

We here collect some proofs omitted in the main text of the paper
\cite{S:AleCri_new20}.

\subsection{Proof of Theorem~\ref{th-chi-squared-test}}
\label{proof-chi-squared-test}
The proof is based on Proposition~\ref{CLT-stesso-ordine} (for case a))
and Theorem~\ref{CLT} (for case b)).  
The almost sure convergence of $O_i/N$ immediately
follows since $O_i/N=\overline{\xi}_{N\,i}$. In order to prove the
stated convergence in distribution, we mimic the classical proof for
the Pearson chi-squared test based on the Sherman Morison formula (see
\cite{S:SM50}), but see also \cite[Corollary~2]{S:RaoScott81}.  

We start recalling the Sherman Morison formula: if $A$ is an
invertible square matrix and we have $ 1 - \bs{v}^\top A^{-1} \bs{u}
\neq 0 $, then
\[
(A - \bs{u}\bs{v}^\top)^{-1} =  
A^{-1} + \frac{A^{-1} \bs{u}\bs{v}^\top A^{-1}}{1 - \bs{v}^\top A^{-1} \bs{u} }.
\]
Given the observation
$\bs{\xi_{n}}=(\xi_{n\,1},\dots,\xi_{n\,k})^{\top}$, we define the
``truncated'' vector $\bs{\xi^*_{n}}= (\xi^*_{n\, 1}, \ldots,
\xi^*_{n\, k-1})^\top$, given by the first $k-1$ components of
$\bs{\xi_{n}}$. Proposition~\ref{CLT-stesso-ordine} (for case a)) and
Theorem~\ref{CLT} (for case b)) give the second order
asymptotic behaviour of $(\bs{\xi_{n}})$, that immediately implies
\begin{equation}\label{eq:truncCLT}
N^e\left(\bs{\overline{\xi}^*_{N}}-\bs{p^*}\right)
=
\frac{\sum_{n=1}^N (\bs{\xi^*_n}  - \bs{p^*}) }{N^{1-e}} 
\mathop{\longrightarrow}^{d} 
\mathcal{N} (\bs{0},\Gamma_*),
\end{equation}
where $\bs{p^*}$ is given by the first $k-1$ components of $\bs{p_0}$
and $ \Gamma_*= \lambda ( \mathrm{diag} (\bs{p^*}) -
\bs{p^*}\bs{p^*}^T)$.  By assumption $p_{0\,i}>0$ for all
$i=1,\dots,k$ and so $\mathrm{diag} (\bs{p^*})$ is invertible with
inverse $\mathrm{diag} (\bs{p^*})^{-1} = \mathrm{diag}
(\frac{1}{p_{0\,1}}, \ldots, \frac{1}{p_{0\,k-1}} )$ and, since $
(\mathrm{diag} (\bs{p^*})^{-1} ) \bs{p^*} =
\bs{1}\in\mathbb{\R}^{k-1}$, we have
\[
1 -  \bs{p^*}^T \mathrm{diag} (\bs{p^*})^{-1} \bs{p^*} = 
1-\sum_{i=1}^{k-1}p_{0\,i}=
\sum_{i=1}^k p_{0\,i} - \sum_{i=1}^{k-1} p_{0\,i} = p_{0\,k} >0.
\]
Therefore we can use the Sherman Morison formula with $A =
\mathrm{diag} (\bs{p^*})$ and $\bs{u}= \bs{v} = \bs{p^*}$, and we
obtain
\begin{equation}\label{eq:SMformula}
(\Gamma_*)^{-1} = \frac{1}{\lambda} 
( \mathrm{diag} (\bs{p^*}) - \bs{p^*}\bs{p^*}^T)^{-1}
=
\frac{1}{\lambda} 
\Big( \mathrm{diag} (\tfrac{1}{p_{0\,1}}, \ldots, \tfrac{1}{p_{0\,k-1}} ) 
+ \frac{1}{p_{0\,k}} \bs{1}\bs{1}^\top \Big).
\end{equation}
Now, since $\sum_{i=1}^{k} (\overline{\xi}_{N\,i} - {p_{0\,i}}) = 0$, 
then $\overline{\xi}_{N\,k} - {p_{0\,k}} = \sum_{i=1}^{k-1}
(\overline{\xi}_{N\,i} - {p_{0\,i}})$ and so we get
\[
\begin{aligned}
\sum_{i=1}^{k} \frac{(O_{i} - N{p_{0\,i}})^2}{N{p_{0\,i}}} & = 
N \sum_{i=1}^{k} \frac{(\overline{\xi}_{N\,i} - {p_{0\,i}})^2}{{p_{0\,i}}} 
= 
N \Big[ 
\sum_{i=1}^{k-1} \frac{(\overline{\xi}_{N\,i} - {p_{0\,i}})^2}{{p_{0\,i}}} + 
\frac{(\overline{\xi}_{N\,k} - {p_{0\,k}})^2}{{p_{0\,k}}} \Big]
\\
& = 
N \Big[ \sum_{i=1}^{k-1} \frac{(\overline{\xi}_{N\,i} - {p_{0\,i}})^2}{{p_{0\,i}}} + 
\frac{(\sum_{i=1}^{k-1} (\overline{\xi}_{N\,i} - {p_{0\,i}}) )^2}{{p_{0\,k}}} 
\Big]
\\
& = 
N \sum_{i_1,i_2=1}^{k-1} (\overline{\xi}_{N\,{i_1}} - {p_{0\,i_1}}) 
(\overline{\xi}_{N\,i_2} - {p_{0\,i_2}}) 
\Big( I_{i_1,i_2}\frac{1}{{p_{0\,i_1}}} + 
\frac{1}{{p_{0\,k}}} \Big),
\end{aligned}
\]
where $I_{i_1\,i_2}$ is equal to $1$ if $i_1=i_2$ and equal to
zero otherwise.  Finally, from the above equalities, recalling
\eqref{eq:truncCLT} and \eqref{eq:SMformula}, we obtain
\begin{equation*}
\frac{1}{N^{1-2e}}\sum_{i=1}^{k} \frac{(O_{i} - N{p_{0\,i}})^2}{N{p_{0\,i}}} =
\lambda 
N^{2e} (\bs{\overline{\xi}^*_N}  - \bs{p^*})^\top 
(\Gamma_*)^{-1} (\bs{\overline{\xi}^*_N}  - \bs{p^*})
\mathop{\longrightarrow}\limits^{d}
\lambda W_0=W_*,
\end{equation*}
where $1-2e\geq 0$ and $W_0$ is a random variable with distribution
$\chi^2(k-1)=\Gamma((k-1)/2,1/2)$, where $\Gamma(a,b)$ denotes the
Gamma distribution with density function
$$
f(w)=\frac{b^a}{\Gamma(a)}w^{a-1}e^{-bw}.
$$
As a consequence, $W_*$ has distribution $\Gamma((k-1)/2,1/(2\lambda))$.

%%%%%%%%%%%%%%%%%%%%%%%%%%%%%%%%%%%%%%%%%%%%%%%%%%%%%%%

\subsection{A preliminary central limit theorem}
The following preliminary central limit theorem is useful for the
proofs of the other central limit theorems stated in \cite{S:AleCri_new20} and in 
Section \ref{sec-compl}.

\begin{theorem}\label{preliminary-CLT}
If 
\begin{equation}\label{ipotesi-base}
\frac{1}{N}\sum_{n=1}^N \mathrm{diag} ({\bs{\psi_{n-1}}}) - {\bs{\psi_{n-1}}}
{\bs{\psi_{n-1}}^\top}
\mathop{\longrightarrow}\limits^{P} 
V\,,
\end{equation}
where $V$ is a random variable with values in the space of positive
semidefinite $k\times k$-matrices, then 
$$
\sqrt{N}\left(\bs{\overline{\mu}_N}-\bs{\overline{\theta}_{N-1}}\right)=
\sqrt{N}\left(\bs{\overline{\xi}_N}-\bs{\overline{\psi}_{N-1}}\right)
\stackrel{s}\longrightarrow \mathcal{N}\left(\bs{0},V\right).
$$
\end{theorem}

\begin{proof} 
We can write
\begin{equation*}
\begin{split}
\sqrt{N}\left(\bs{\overline{\xi}_N}-\bs{\overline{\psi}_{N-1}}\right)
&=\frac{1}{\sqrt{N}} 
N\left(\bs{\overline{\xi}_N}-\bs{\overline{\psi}_{N-1}}\right)
=\frac{1}{\sqrt{N}} \sum_{n=1}^N(\bs{\xi_n}-\bs{\psi_{n-1}})
\\
&=\frac{1}{\sqrt{N}}\sum_{n=1}^N\Delta\bs{M_n}=\sum_{n=1}^N Y_{N,n},
\end{split}
\end{equation*}
with $\bs{Y_{N,n}}=N^{-1/2}\Delta\bs{M_n}$.  For the convergence of
$\sum_{n=1}^N\bs{Y_{N,n}}$, we observe that
$E[\bs{Y_{N,k}}|\mathcal{F}_{k-1}]=\bs{0}$ and so, by Theorem
\ref{stable-conv},  
it converges stably to
$\mathcal{N}(\bs{0},V)$ if the conditions (c1) and (c2) hold true.
Regarding (c1), we note that $\max_{1\leq n\leq N} |\bs{Y_{N,n}} |
\leq \frac{1}{\sqrt{N}}\max_{1\leq n\leq N} |\bs{\xi_{n}}
-{\bs{\psi_{n-1}}}|=O(1/\sqrt{N}) \to 0$.  Condition (c2)
means
$$
\sum_{n=1}^N \bs{Y_{N,n}} {\bs{Y}_{\bs{N,n}}^\top}
=
\frac{1}{N}\sum_{n=1}^N
(\bs{\xi_{n}}-{\bs{\psi_{n-1}}})(\bs{\xi_{n}}-{\bs{\psi_{n-1}}})^\top
\stackrel{P}\longrightarrow V.
$$ The above convergence holds true by Assumption \eqref{ipotesi-base}
and Lemma \ref{lemma-serie-rv} (with $c_n=n$ and $v_{N,n}=n/N$). % in appendix. 
Indeed, we have $\sum_{n\geq 1}
E[\|\bs{\xi_{n}}-{\bs{\psi_{n-1}}}\|^2]/n^2\leq \sum_{n\geq 1}
n^{-2}<+\infty$ and
$$
E[
(\bs{\xi_{n}}-{\bs{\psi_{n-1}}})(\bs{\xi_{n}}-{\bs{\psi_{n-1}}})^\top
|\mathcal{F}_{n-1}]=
\mathrm{diag} ({\bs{\psi_{n-1}}}) - {\bs{\psi_{n-1}}}
{\bs{\psi_{n-1}}^\top}\,.
$$
\end{proof}

\begin{remark}\label{remark-preliminary-CLT}
\rm Recalling that $\bs{\psi_n}=\bs{\theta_n}+\bs{p_0}$, the
convergence \eqref{ipotesi-base} with $V=\Gamma=\mathrm{diag}
({\bs{p_0}}) - {\bs{p_0}}{\bs{p_0}^\top}$, means
\begin{equation*}%\label{cond-conv-medie}
\bs{\overline{\theta}_{N-1}}=
\frac{1}{N}\sum_{n=1}^N\bs{\theta_{n-1}}\stackrel{P}\longrightarrow\bs{0}
\qquad\mbox{and}\qquad
\frac{1}{N}\sum_{n=1}^N\bs{\theta_{n-1}}\bs{\theta_{n-1}}^{\top}
\stackrel{P}\longrightarrow 0_{k\times k}\,,
\end{equation*}
 where $0_{k\times k}$ is the null matrix with dimension $k\times
k$.
\end{remark}

%%%%%%%%%%%%%%%%%%%%%%%%%%%%%
\subsection{Proof of Proposition \ref{CLT-stesso-ordine}}
\label{proof-CLT-stesso-ordine}
By Lemma \ref{lemma-serie-rv} (with $c_n=n$ and $v_{N,n}=n/N$), Remark
\ref{case-as-conv} and Theorem \ref{SA-conv}, 
we
immediately get $\bs{\overline{\xi}_N}\to\bs{p_0}$ almost
surely. Indeed, we have
$E[\bs{\xi_{n+1}}|\mathcal{F}_n]=\bs{\psi_n}\to \bs{p_0}$ almost
surely and $\sum_{n\geq 1} E[\|\bs{\xi_n}\|^2 ]n^{-2}\leq \sum_{n\geq
  1} n^{-2}<+\infty$.  \\ \indent Regarding the central limit theorem
for $\bs{\overline{\xi}_N}$, we have to distinguish the two cases
$1/2< \epsilon\leq 1$ or $0<\epsilon\leq 1/2$.  In the first case, the
result follows from Theorem \ref{SA-CLT-mu}, 
because \eqref{dinamica-mun} and the fact that
$E[\Delta\bs{M_{n+1}}\Delta\bs{M_{n+1}}^{\top}\,|\mathcal{F}_n]=
\mathrm{diag} ({\bs{\psi_{n-1}}}) - {\bs{\psi_{n-1}}}
       {\bs{\psi_{n-1}}^\top} \to \Gamma$ almost surely; while for the
       second case the result follows from Theorem
       \ref{preliminary-CLT}.  Indeed, we have
\begin{equation*}
\begin{split}
\sqrt{N}\left(\bs{\overline{\xi}_N}-\bs{p_0}\right)
&=
\sqrt{N}\left(\bs{\overline{\xi}_N}-\bs{\overline{\psi}_{N-1}}\right)+
\sqrt{N}\left(\bs{\overline{\psi}_{N-1}}-\bs{p_0}\right)
\\
&=
(c+1)\sqrt{N}\left(\bs{\overline{\xi}_N}-\bs{\overline{\psi}_{N-1}}\right)
-\sqrt{N}\bs{D_N},
\end{split}
\end{equation*}
where ${\bs D_N}=
c\left(\bs{\overline{\xi}_N}-\bs{\overline{\psi}_{N-1}}\right)-
\left(\bs{\overline{\psi}_{N-1}}-\bs{p_0}\right)$.  By Theorem
\ref{preliminary-CLT}, the term
$(c+1)\sqrt{N}\left(\bs{\overline{\xi}_N}-\bs{\overline{\psi}_{N-1}}\right)$
stably converges to $\mathcal{N}(0, (c+1)^2\Gamma)$ (note that assumption
\eqref{ipotesi-base} is satisfied with $V=\Gamma$, because
$\bs{\psi_n}\to\bs{p_0}$ almost surely).  Therefore, in order to
conclude, it is enough to show that $\sqrt{N}\bs{D_N}$ converges in
probability to $\bs{0}$. To this purpose, we observe that, by
\eqref{dinamica-psin-bis} with $\delta_n=c\epsilon_n$, we have
\begin{equation*}
\bs{\psi_n}-\bs{\psi_{n-1}}=
\epsilon_{n-1}
\left[c(\bs{\xi_{n}}-\bs{\psi_{n-1}})-(\bs{\psi_{n-1}}-\bs{p_0})\right]
\end{equation*}
and so 
\begin{equation*}
\begin{split}
{\bs D_N}=
\frac{1}{N}\sum_{n=1}^{N}\frac{\bs{\psi_n}-\bs{\psi_{n-1}}}{\epsilon_{n-1}}.
\end{split}
\end{equation*}
Moreover, we note that $
\sum_{n=1}^{+\infty}(\bs{\psi_n}-\bs{\psi_{n-1}})=
\lim_N\bs{\psi_N}-\bs{\psi_0}=\bs{p_0}-\bs{\psi_0}<+\infty$ and, by
Lemma \ref{gen-kro-lemma}  
(with
$v_{N,n}=\epsilon_{N-1}/\epsilon_{n-1}$), we get
$$
\epsilon_{N-1}\sum_{n=1}^{N}\frac{\bs{\psi_n}-\bs{\psi_{n-1}}}{\epsilon_{n-1}}
\stackrel{a.s.}\longrightarrow \bs{0}.
$$
For $\epsilon\leq 1/2$, this fact implies  
$$
\sqrt{N}\bs{D_N}=\frac{1}{\sqrt{N}\epsilon_{N-1}}
\epsilon_{N-1}\sum_{n=1}^{N-1}\frac{\bs{\psi_n}-\bs{\psi_{n-1}}}{\epsilon_{n-1}}
\stackrel{a.s.}\longrightarrow \bs{0}\,.
$$
The proof is thus concluded.
\qed

%%%%%%%%%%%%%%%%%%%%%%%%%%%%
\section{Case $\sum_n\epsilon_n<+\infty$}\label{sec-compl}

In this section we provide some results regarding the case
$\sum_n\epsilon_n<+\infty$, even if, as we will see, this case is not
interesting for the chi-squared test of goodness of fit. Indeed, as
shown in the following result, the empirical mean almost surely
converges to a random variable, which does not coincide almost surely
with a deterministic vector.

\begin{theorem}
If $\sum_{n=0}^{+\infty} \epsilon_n<+\infty$, then
$\bs{\overline{\xi}_N}\stackrel{a.s.}\longrightarrow\bs{\psi_\infty}$,
where $\bs{\psi_\infty}$ is a random variable, which is not almost
surely equal to a deterministic vector, that is
$P(\bs{\psi_\infty}\neq \bs{q_0})>0$ for all $\bs{q_0}\in
\mathbb{R}^k$.
\end{theorem}

\begin{proof} When $\sum_{n=0}^{+\infty}\epsilon_n<+\infty$, the sequence
$(\bs{\psi_n})$ is a (bounded) non-negative almost supermartingale
  (see \cite{S:rob}) because, by \meqref{dinamica-psin-bis}, we have
$$
E[\bs{\psi_{n+1}}|\mathcal{F}_n]=  
\bs{\psi_n}(1-\epsilon_n)+\epsilon_n\bs{p_0}
\leq
\bs{\psi_n}+\epsilon_n \bs{p_0}.
$$ As a consequence, it converges almost surely (and in $L^p$ with
$p\geq 1$) to a certain random variable $\bs{\psi_\infty}$. An
alternative proof of this fact follows from quasi-martingale theory
\cite{S:met}: indeed, since $\sum_n E[\,
  |E[\bs{\psi_{n+1}}|\mathcal{F}_n]-\bs{\psi_n}|\, ]
=O(\sum_n\epsilon_n)<+\infty$, the stochastic process $(\bs{\psi_n})$
is a non-negative quasi-martingale and so it converges almost surely
(and in $L^p$ with $p\geq 1$) to a certain random variable
$\bs{\psi_\infty}$.  \\ \indent The almost sure convergence of
$\bs{\overline{\xi}_n}$ to $\bs{\psi_\infty}$ follows by Lemma
\ref{lemma-serie-rv} and Remark \ref{case-as-conv} (with
$c_n=n$ and $v_{N,n}=n/N$), because
$E[\bs{\xi_{n+1}}|\mathcal{F}_n]=\bs{\psi_n}\to \bs{\psi_\infty}$
almost surely and $\sum_{n\geq 1} E[\|\bs{\xi_n}\|^2 ]n^{-2}\leq
\sum_{n\geq 1} n^{-2}<+\infty$.  \\ \indent In order to show that
$\bs{\psi_\infty}$ is not almost surely equal to a deterministic
vector, we set
$$
y_n=E[\|\bs{\psi_{n}}-\bs{p_0}\|^2]-\|E[\bs{\psi_n}-\bs{p_0}]\|^2
=\sum_{i=1}^k Var[\psi_{n\,i}-p_{0\,i}]
$$
and observe that, starting from \meqref{dinamica-psin-bis}, we get
$$
\bs{\psi_{n+1}}-\bs{p_0}=(1-\epsilon_n)(\bs{\psi_n}-\bs{p_0})
+\delta_n\Delta\bs{M_{n+1}}
$$
and so 
\begin{equation*}
\|E[\bs{\psi_n}-\bs{p_0}]\|^2
=E[\bs{\psi_n}-\bs{p_0}]^{\top}E[\bs{\psi_n}-\bs{p_0}]
=(1-\epsilon_n)^2\|E[\bs{\psi_n}-\bs{p_0}]\|^2
\end{equation*}
and
\begin{equation*}
  \begin{split}
    E[\|\bs{\psi_{n+1}}-\bs{p_0}\|^2]
&=E[(\bs{\psi_{n+1}}-\bs{p_0})^{\top}(\bs{\psi_{n+1}}-\bs{p_0})]
\\    
&=
(1-\epsilon_n)^2E[\|\bs{\psi_n}-\bs{p_0}\|^2]
+\delta_n^2E[\|\Delta\bs{M_{n+1}}\|^2]\,.
\end{split}
  \end{equation*}
Hence,  we obtain 
\begin{equation}\label{dinamica-yn}
  y_{n+1}=(1-\epsilon_n)^2y_n+\delta_n^2E[\|\Delta\bs{M_{n+1}}\|^2]
=(1-2\epsilon_n)y_n+\widetilde{\zeta}_n
\end{equation}
with $\widetilde{\zeta}_n=\epsilon_n^2 y_n+\delta_n^2E[\|\Delta
  \bs{M_{n+1}}\|^2]\geq 0$.  It follows that, given $\tilde{n}$ such that
$\epsilon_n<1/2$ for $n\geq \tilde{n}$, we have $y_N\geq
y_{\tilde{n}}\prod_{n=\tilde{n}}^{N-1}(1-2\epsilon_n)$ for each $N\geq
\tilde{n}$ and so
$$
E[\|\bs{\psi_\infty}-\bs{p_0}\|^2]-\|E[\bs{\psi_\infty}-\bs{p_0}]\|^2
=y_{\infty}
=\lim_{N\to +\infty}y_N\geq y_{\tilde{n}}\prod_{n=\tilde{n}}^{+\infty}(1-2\epsilon_n)
=y_{\tilde{n}}\exp\left(\sum_{n=\tilde{n}}^{+\infty}\ln(1-2\epsilon_n)\right)\,.
$$ The above exponential is strictly greater than $0$ because
$\sum_{n=\tilde{n}}^{+\infty}\ln(1-2\epsilon_n)\sim
-2\sum_{n=\tilde{n}}^{+\infty}\epsilon_n>-\infty$. Therefore, if
$y_{\tilde{n}}>0$, then we have $y_\infty>0$. This means that
$\bs{\psi_\infty}-\bs{p_0}$, and consequently $\bs{\psi_\infty}$, is
not almost surely equal to a deterministic vector, that is
$P(\bs{\psi_\infty}\neq \bs{q_0})>0$ for all $\bs{q_0}\in
\mathbb{R}^k$. If $y_{\tilde{n}}=0$, that is if
$\bs{\psi_{\tilde{n}}}$ is almost surely equal to a deterministic
vector $\bs{\widetilde{\psi}}$, then, by \eqref{dinamica-yn}, we get
$$
y_{\tilde{n}+1}=\delta_n^2E[\|\Delta \bs{M_{n+1}}\|^2]
=\delta_{\tilde{n}}^2E[\|\bs{\xi_{\tilde{n}+1}}-\bs{\widetilde{\psi}}\|^2]>0\,, 
$$ because $\delta_n>0$ for each $n$ and $\bs{\widetilde{\psi}}$ is
different from a vector of the canonical base of $\mathbb{R}^k$ by
means of the assumption $b_{0\,i}+B_{0\,i}>0$ and equality
\meqref{eq:extract1a:K-vettoriale}. It follows that we can repeat the
above argument replacing $\tilde{n}$ by $\tilde{n}+1$ and conclude
that $\bs{\psi_\infty}$ is not almost surely equal to a deterministic
vector.
\end{proof}

As a consequence of the above theorem, if we aim at having the almost
sure convergence of $\bs{\overline{\xi}_N}$ to a deterministic vector, 
we have to avoid the case $\sum_{n=0}^{+\infty}
\epsilon_n<+\infty$. However, for the sake of completeness, we provide
a second-order convergence result also in this case. First, we note
that Theorem \ref{preliminary-CLT} still holds true with
$V=\mathrm{diag} ({\bs{\psi_{\infty}}}) - {\bs{\psi_{\infty}}}
{\bs{\psi_{\infty}}^\top}$. Indeed, assumption \eqref{ipotesi-base} is
satisfied by Lemma \ref{lemma-serie-rv} and Remark \ref{case-as-conv}
(with $c_n=n$ and $v_{N,n}=n/N$), because of the almost
sure convergence of $\bs{\psi_n}$ to $\bs{\psi_\infty}$. Moreover, we
have the following theorem:
\begin{theorem}\label{CLT-serie-finita}
Suppose to be in one of the following two cases:
\begin{itemize}
\item[a)] $\sum_{n=1}^N n\epsilon_{n-1}=o(\sqrt{N})$ and
  $\sum_{n=1}^N n\delta_{n-1}=o(\sqrt{N})$; 
\item[b)] $\epsilon_n=(n+1)^{-\epsilon}$ and $\delta_n\sim c
  (n+1)^{-\delta}$ with $c>0$, $\delta\in (1/2,1)$ and
  $\epsilon>\delta+1/2$ ($\epsilon=+\infty$ included, that means
  $\epsilon_n=0$ for all $n$).
\end{itemize}
Set $e=1/2$ and $\lambda=1$ in case a) and $e=\delta-1/2\in (0,1/2)$
and $\lambda=c^2/[2(1-e)]=c^2/(3-2\delta)$ in case b). Then, we have
$$
  N^e\left(\bs{\overline{\xi}_{N}}-\bs{\psi_{N}}\right)
  \stackrel{s}\longrightarrow\mathcal{N}\left(0,\lambda \Gamma\right)\,,
$$ where $\Gamma=\mathrm{diag} ({\bs{\psi_{\infty}}}) -
           {\bs{\psi_{\infty}}} {\bs{\psi_{\infty}}^\top}$.\\ \indent
           When
           $\left(\bs{\psi_{N}}-\bs{\psi_{\infty}}\right)=o_P(N^{-e})$,
           we also have
$$
N^e\left(\bs{\overline{\xi}_{N}}-\bs{\psi_{\infty}}\right)
\stackrel{s}\longrightarrow\mathcal{N}\left(0,\lambda\Gamma\right).
$$
\end{theorem}
Note that case a) covers the case $\epsilon_n=(n+1)^{-\epsilon}$ and
$\delta_n\sim c(n+1)^{-\delta}$ with $c>0$ and
$\min\{\epsilon,\delta\}>3/2$.  \\ \indent The case $\epsilon_n=0$
(that is $\beta_n=1$) for all $n$ corresponds to the case considered
in \cite{S:pem90}, but in that paper the author studies only the limit
$\bs{\psi_\infty}$ and he does not provide second-order convergence
results.  

\begin{proof} We have
$$
\begin{aligned}
N^e\left(\bs{\overline{\xi}_{N}}-\bs{\psi_{N}}\right) &=
\frac{1}{N^{1-e}}\left(N\bs{\overline{\xi}_{N}}-N\bs{\psi_{N}}\right)
=
\frac{1}{N^{1-e}}\sum_{n=1}^N\left[\bs{\xi_{n}}-\bs{\psi_{n-1}}
+n({\bs{\psi_{n-1}}}-{\bs{\psi_{n}}})\right]
\\ &
\frac{1}{N^{1-e}}\sum_{n=1}^N(\bs{\xi_{n}}-\bs{\psi_{n-1}})+
\frac{1}{N^{1-e}}\sum_{n=1}^N n\epsilon_{n-1}(\bs{\psi_{n-1}}-\bs{p_0})-
\frac{1}{N^{1-e}}\sum_{n=1}^N n\delta_{n-1}\Delta\bs{M_{n}}
\\ &
=\frac{1}{N^{1/2-e}}\sum_{n=1}^N \bs{Y_{N,n}}+\sum_{n=1}^N\bs{Z_{N,n}}+\bs{Q_N},
\end{aligned}
$$
where
$$
\bs{Y_{N,n}}=
\frac{\bs{\xi_{n}}-{\bs{\psi_{n-1}}}}{\sqrt{N}}=
\frac{\Delta\bs{M_{n}}}{\sqrt{N}},
\qquad
\bs{Z_{N,n}}=-
\frac{n\delta_{n-1}(\bs{\xi_{n}}-{\bs{\psi_{n-1}}})}{N^{1-e}}=
\frac{n\delta_{n-1}\Delta\bs{M_{n}}}{N^{1-e}}
$$
and
$$
\bs{Q_N}=
\frac{1}{N^{1-e}}\sum_{n=1}^N 
n\epsilon_{n-1} (\bs{\psi_{n-1}}-\bs{p_0}).
$$ In both cases a) and b), we have $\sum_{n=1}^N
n\epsilon_{n-1}=o(N^{1-e})$ and so $\bs{Q_N}$ converges almost surely to
$\bs{0}$. Moreover, by Theorem \ref{preliminary-CLT},
$\sum_{n=1}^N\bs{Y_{N,n}}$ stable converges to $\mathcal{N}(\bs{0},V)$
with $V=\Gamma=\mathrm{diag} ({\bs{\psi_{\infty}}}) -
{\bs{\psi_{\infty}}} {\bs{\psi_{\infty}}^\top}$.  Therefore it is
enough to study the convergence of $\sum_{n=1}^N\bs{Z_{N,n}}$.  To
this purpose, we observe that, if we are in case a), then
$\sum_{n=1}^N \bs{Z_{N,n}}$ converges almost surely to $\bs{0}$ and so
$$
\sqrt{N}\left(\bs{\overline{\xi}_{N}}-\bs{\psi_{N}}\right)
\stackrel{s}\longrightarrow \mathcal{N}(0,\Gamma).
$$ Otherwise, if we are in case b), we observe that
$E[\bs{Z_{N,n}}|\mathcal{F}_{n-1}]=\bs{0}$ and so
$\sum_{n=1}^N\bs{Z_{N,n}}$ converges stably to
$\mathcal{N}(\bs{0},\lambda \Gamma)$ if the conditions (c1) and (c2)
of Theorem \ref{stable-conv}, with $V=\lambda \Gamma$,
hold true.  Regarding (c1), we observe that $\max_{1\leq n\leq N}
|\bs{Z_{N,n}} | \leq \frac{1}{N^{1-e}}\max_{1\leq n\leq N}
n\delta_{n-1}|\bs{\xi_{n}} -{\bs{\psi_{n-1}}}|=O(1/\sqrt{N})$.
Regarding condition (c2), that is
$$
\sum_{n=1}^N \bs{Z_{N,n}} {\bs{Z_{N,n}}^\top}
=
\frac{1}{N^{2(1-e)}}\sum_{n=1}^N n^2\delta_{n-1}^2
(\bs{\xi_{n}}-{\bs{\psi_{n-1}}})(\bs{\xi_{n}}-{\bs{\psi_{n-1}}})^\top
\mathop{\longrightarrow}\limits^{P} \frac{c^2}{2(1-e)}\Gamma\,,
$$ we observe that it holds true even almost surely, because $
\frac{1}{N^{2(1-e)}}\sum_{n=1}^N n^2\delta_{n-1}^2\to
c^2/[2(1-e)]=c^2/(3-2\delta)$ and
$$
E[
(\bs{\xi_{n}}-{\bs{\psi_{n-1}}})(\bs{\xi_{n}}-{\bs{\psi_{n-1}}})^\top
|\mathcal{F}_{n-1}]=
\mathrm{diag} ({\bs{\psi_{n-1}}}) - {\bs{\psi_{n-1}}}
{\bs{\psi_{n-1}}^\top}\stackrel{a.s.}\longrightarrow
\Gamma
$$ (see Lemma \ref{lemma-serie-rv} and Remark \ref{case-as-conv} with
$c_n=n$ and $v_{N,n}=n^3\delta_{n-1}^2/N^{2(1-e)}\sim
c^2(n/N)^{3-2\delta}$).  Therefore, we have
$$
N^e\left(\bs{\overline{\xi}_{N}}-\bs{\psi_{N}}\right)
\stackrel{s}\longrightarrow
\mathcal{N}\left(0,c^2(3-2\delta)^{-1}\Gamma\right).
$$ 
Finally, we observe that 
$$
N^e\left(\bs{\overline{\xi}_{N}}-\bs{\psi_{\infty}}\right)=
N^e\left(\bs{\overline{\xi}_{N}}-\bs{\psi_{N}}\right)+
N^e\left(\bs{\psi_{N}}-\bs{\psi_{\infty}}\right).
$$ Therefore, when
$\left(\bs{\psi_{N}}-\bs{\psi_{\infty}}\right)=o_P(N^{-e})$,
we have
$$
N^e\left(\bs{\overline{\xi}_{N}}-\bs{\psi_{\infty}}\right)
\stackrel{s}\longrightarrow
\mathcal{N}\left(0,\lambda \Gamma\right).
$$
\end{proof}

\indent An example of the case a) of Theorem \ref{CLT-serie-finita} 
with $\left(\bs{\psi_{N}}-\bs{\psi_{\infty}}\right)=o_P(N^{-e})$ is
the RP urn with $\alpha_n=\alpha>0$ and $\beta_n=\beta>1$ (see
\cite{S:ale-cri-RP}).  Indeed, in this case, we have $\epsilon_n\sim
c_\epsilon\beta^{-n}$ and $\delta_n\sim c_\delta \beta^{-n}$, where
$c_\epsilon>0$ and $c_\delta>0$ are suitable constants, and
$\left(\bs{\psi_{N}}-\bs{\psi_{\infty}}\right)=O(\beta^{-N})$. We
conclude this section with other two examples regarding the case
$\epsilon_n=0$ (that is $\beta_n=1$) for all $n$.

\begin{example} {\em (Case $\epsilon_n=0$ and 
$\delta_n\sim c(n+1)^{-\delta}$ with $c>0$ and $\delta>3/2$)}\\ \rm If
  $\epsilon_n=0$ for all $n$, then we have
  $r_n^*=|\bs{b_0}|+|\bs{B_0}|+\sum_{h=1}^n\alpha_h$. Therefore, if we
  take $\alpha_n=n^{-\delta}$, with $\delta>3/2$, then $r_n^*$
  converges to the constant
  $r^*=|\bs{b_0}|+|\bs{B_0}|+\sum_{h=1}^{+\infty}h^{-\delta}$ and
  $\delta_n=\alpha_{n+1}/r_{n+1}^*\sim
  c\alpha_{n+1}=c(n+1)^{-\delta}$, with $c=1/r^*$. Moreover, since
  $\delta>3/2$, assumption a) of Theorem \ref{CLT-serie-finita} is
  satisfied. We also observe that $\sum_n\delta_n^2<+\infty$ and so
  $\psi_{\infty\,i}$ is not concentrated on $\{0,1\}$ and has no atoms
  in $(0,1)$ (see \cite[Th.~2 and Th.~3]{S:pem90}). More precisely, we
  have
$$
\bs{\psi_\infty}=\frac{\bs{b_0}+\bs{B_0}+\sum_{n=1}^{+\infty}\alpha_n\bs{\xi_n}}
{|\bs{b_0}|+|\bs{B_0}|+\sum_{n=1}^{+\infty}\alpha_n}
$$
and so 
\begin{equation*}
\begin{split}
&\bs{\psi_N}-\bs{\psi_\infty}=
\\
&\frac{
(\bs{b_0}+\bs{B_0}+\sum_{n=1}^{N}\alpha_n\bs{\xi_n})\sum_{n\geq N+1}\alpha_n
-
(|\bs{b_0}|+|\bs{B_0}|+\sum_{n=1}^{N}\alpha_n)\sum_{n\geq N+1}\alpha_n\bs{\xi_n}
}
{
(|\bs{b_0}|+|\bs{B_0}|+\sum_{n=1}^{N}\alpha_n)
(|\bs{b_0}|+|\bs{B_0}|+\sum_{n=1}^{+\infty}\alpha_n)
}=
\\
&O\left(\sum_{n\geq N+1}\alpha_n\right)=O\left(N^{1-\delta}\right).
\end{split}
\end{equation*}
Since $\delta>3/2$, we get
$(\bs{\psi_N}-\bs{\psi_\infty})=o(N^{-1/2})$. This fact can also be
obtained as a consequence of Theorem \ref{thm:clt-psi-epsilon-zero}
below. Indeed, this theorem states that the rate of convergence of
$\bs{\psi_N}$ to $\bs{\psi_\infty}$ is $N^{-(\delta-1/2)}$.
\\ \indent Note that, since $\beta_n=1$ for all $n$, the factor
$f(h,n)$ in \meqref{eq-psi_n} coincides with $\alpha_h$ and so, in this
case, it is decreasing.
\end{example}

\begin{example} {\em (Case $\epsilon_n=0$ and 
$\delta_n\sim c(n+1)^{-\delta}$ with $c>0$ and $\delta\in (1/2,1)$)}
  \\ \rm As in the previous example, since $\epsilon_n=0$ for all $n$,
  we have $r_n^*=|\bs{b_0}|+|\bs{B_0}|+\sum_{h=1}^n\alpha_h$.  Let us
  set $A_n=\sum_{h=1}^n\alpha_h=\exp(b n^\alpha)$ with $b>0$ and
  $\alpha\in(0,1/2)$, which brings to $r_n^*\sim A_n\uparrow +\infty$
  and $\alpha_n=\exp(b n^\alpha)-\exp(b (n-1)^\alpha) $ and
\begin{equation*}
\begin{split}
\delta_{n-1}&=\frac{\alpha_n}{|\bs{b_0}|+|\bs{B_0}|+A_n}\sim
1-\frac{\sum_{h=1}^{n-1}\alpha_h}{\sum_{h=1}^n\alpha_h}\\
&=1-\exp\left[b\left((n-1)^{\alpha}-n^{\alpha}\right)\right]\\
&=
b
n^{\alpha}\left( 1- (1-n^{-1})^\alpha \right)
+O\left(\, n^{2\alpha}( 1-(1-n^{-1})^\alpha )^2 \,\right)
=b n^{\alpha}\left( \alpha n^{-1}+O(n^{-2}) \right)+O(n^{-(2-2\alpha)})\\
&=b \alpha n^{-(1-\alpha)}+O(n^{-(2-\alpha)})+O(n^{-(2-2\alpha)})=
b \alpha n^{-(1-\alpha)}+O(n^{-2(1-\alpha)}),
\end{split}
\end{equation*}
so that $\delta=(1-\alpha)\in (1/2,1)$ and $c=b\alpha>0$.  Hence, we
have $\delta_n\sim c(n+1)^{-\delta}$ and assumption b) of Theorem
\ref{CLT-serie-finita} is satisfied. We also observe that
$\sum_n\delta_n^2<+\infty$ and so $\psi_{\infty\,i}$ is not
concentrated on $\{0,1\}$ and has no atoms in $(0,1)$ (see \cite[Th.~2
  and Th.~3]{S:pem90}). Moreover, by Theorem
\ref{thm:clt-psi-epsilon-zero} below, we get that $N^{e}
\left(\bs{\psi}_{N}-\bs{\psi_{\infty}}\right){\longrightarrow}
\mathcal{N}\left(0,c^2(2e)^{-1}\Gamma\right)$, where
$e=\delta-1/2$. Hence, applying Theorem \ref{blocco}, we
obtain
$$ N^e\left(\bs{\overline{\xi}_{N}}-\bs{\psi_{\infty}}\right)
\stackrel{s}\longrightarrow\mathcal{N}\left(0,c^2[2e(1-e)]^{-1}\Gamma\right).
$$ Finally, note that, as before, since $\beta_n=1$ for all $n$, the
factor $f(h,n)$ in \meqref{eq-psi_n} coincides with $\alpha_h$ and so,
in this case, $\ell(h) = \ln(f(h,n))=\ln(\alpha_h) \sim
\ln(\delta_{h-1})+bh^\alpha\sim bh^\alpha - b\alpha(1-\alpha)\ln(h)$.
Hence, there exists $h^*$ such that $h\mapsto \ell(h)$ is increasing
for $h\geq h^*$. Since $\max_{h\leq h^*}\ell(h)\leq C$, for a suitable
constant $C$, the contributions of the observations until $h^*$ are
eventually smaller than those with $h \geq h^*$, that are increasing
with $h$.
\end{example}

\begin{theorem}\label{thm:clt-psi-epsilon-zero}
For $\epsilon_n=0$ for all $n$ and $\delta_n\sim c(n+1)^{-\delta}$ with
$c>0$ and $1/2<\delta\leq 1$, we have
\begin{equation*}%\label{eq:clt-psi}
N^{\delta-\frac{1}{2}}
\left(\bs{\psi}_{N}-\bs{\psi_{\infty}}\right){\longrightarrow}
\mathcal{N}
\left(0,c^2(2\delta-1)^{-1}\Gamma\right)\qquad
\mbox{stably in the strong sense w.r.t. } \mathcal{F},
\end{equation*}
where $\Gamma=\mathrm{diag} ({\bs{\psi_{\infty}}}) -
           {\bs{\psi_{\infty}}} {\bs{\psi_{\infty}}^\top}$.
\end{theorem}

\begin{proof} We want to apply Theorem \ref{strong-stable-conv}.
  To this purpose, we recall that, when $\epsilon_n=0$ for all $n$,
  the process $(\bs{\psi_n})$ is a martingale with respect to
  $\mathcal{F}$. Moreover, it converges almost surely and in mean to
  $\bs{\psi_\infty}$. Therefore, in order to conclude, it is enough to
  check conditions (c1) and (c2) of Theorem
  \ref{strong-stable-conv}. Regarding the first condition, we note
  that
  $$ N^{\delta-1/2}\, \sup_{n\geq N} |\bs{\psi_n}-\bs{\psi_{n+1}}|=
  N^{\delta-1/2}\sup_{n\geq N}\delta_n|\Delta\bs{M_{n+1}}|=
  O(N^{\delta-1/2-\delta})=O(N^{-1/2})\longrightarrow 0.
$$
Finally, regarding the second condition, we observe that  
\[
\begin{aligned}
N^{2\delta-1}\sum_{n\geq
  N}(\bs{\psi_n}-\bs{\psi_{n+1}})(\bs{\psi_n}-\bs{\psi_{n+1}})^{\top}
&\sim 
N^{2\delta-1}c^2\sum_{n\geq N}
(n+1)^{-2\delta}(\Delta\bs{M_{n+1}})(\Delta\bs{M_{n+1}})^{\top}
\\
&\stackrel{a.s.}{\longrightarrow} \frac{c^2}{(2\delta-1)}\Gamma,
\end{aligned}\]
where the almost sure convergence follows from \cite[Lemma 4.1]{S:cri-dai-min}
and the fact that 
$$
E[(\Delta\bs{M_{n+1}})(\Delta\bs{M_{n+1}})^{\top}|\mathcal{F}_{n}]=
E[(\bs{\xi_{n+1}}-\bs{\psi_n})(\bs{\xi_{n+1}}-\bs{\psi_n})^{\top}|\mathcal{F}_{n}]
\stackrel{a.s.}{\longrightarrow}\Gamma.
$$
\end{proof}

\section{Computations regarding the local reinforcement}
\label{conti}

Suppose $\alpha_n\sim a n^{-\alpha}$ for $n\geq 1$ and
$(1-\beta_n)\sim b (n+1)^{-\beta}$ for $n\geq 0$.  In the following
subsections we study the behaviour of the factor
$f(h,n)=\alpha_h\prod_{j=h}^{n-1}\beta_j$ in some particular cases
that cover the cases of the two examples in Section \mref{sec-main}.
Specifically, for all the considered cases, we set $\ell(h,n) = \ln(
\alpha_h \prod_{j=h}^{n-1} \beta_j)
=\ln(\alpha_h)+\sum_{j=h}^{n-1}\ln(\beta_j)$ for $n\geq h$ and we
prove that there exists $h_*$ such that $\max_{h\leq h_*}\ell(h,n)
\leq \ell(h_*,n)$ and $h\mapsto \ell(h,n)$ is increasing for $h\geq
h_*$. This means that the weights $f(h,n)$ of the observations until
$h_*$ are smaller than those with $h \geq h_*$ and the contribution of
the observation for $h\geq h_*$ is increasing with $h$.

\subsection{Case $\alpha=\beta\in (0,1)$}
Suppose $\alpha_n=an^{-\alpha}$ and $1-\beta_n=b(n+1)^{-\alpha}$, with
$a,\, b>0$ and $\alpha\in (0,1)$. For $n\geq h$,  we have
\begin{equation*}
\begin{aligned}
  \ell(h+1,n)-\ell(h,n) & = \ln(a(h+1)^{-\alpha }) -\ln(a h^{-\alpha }) 
  -\ln( 1 - b(h+1)^{-\alpha})
\\
& =
-\alpha\ln\Big( 1 + \frac{1}{h} \Big)-\ln\Big( 1 - \frac{b}{(h+1)^\alpha}\Big) 
= - \frac{\alpha}{h} + \frac{b}{(h+1)^\alpha}.  
\end{aligned}
\end{equation*}
Since $\alpha<1$, there exists $h_0$ such that the function $h\mapsto
\ell(h,n)$ is monotonically increasing for $h\geq h_0$. Now, fix
$\eta>0$ and let $j_0$ such that $j\geq j_0$ implies $\ln(\beta_j)
\leq -\frac{b j^{-\alpha}}{1+\eta}$.  Then take
$h_*\geq\max(h_0,j_0)+1$ and $h\leq h_0-1$. For $h_*$ large
enough, we get 
\begin{equation*}
\begin{aligned}
  \ell(h_*,n)-\ell(h,n) & =
  \ln(\alpha_{h^*}) - \ln(\alpha_{h}) -\sum_{j=h}^{h_*-1}\ln(\beta_j)
  =\ln(ah_*^{-\alpha})-\ln(ah^{-\alpha})-\sum_{j=h}^{h_*-1}\ln(\beta_j)
\\
&
\geq
\ln(h_*^{-\alpha}) + \sum_{j=\max(h_0,j_0)}^{h_*-1} \frac{b j^{-\alpha}}{1+\eta}
\\
&
\geq 
-\alpha \ln( h_*)  +
C_1 + \frac{b}{1+\eta} \int_{\max(h_0,j_0)}^{h_*-1} x^{-\alpha} \,dx
\\
& 
= 
-\alpha \ln( h_*)  + C_1 + \frac{b}{(1+\eta)(1-\alpha)} 
\big[ (h_*-1)^{1-\alpha} - \max(h_0,j_0)^{1-\alpha} \big]
\\
& 
= 
C_2 -\alpha \ln( h_*)  + \frac{b}{(1+\eta)(1-\alpha)} (h_*-1)^{1-\alpha}\geq 0\,.
\end{aligned}
\end{equation*}
Therefore, taking $h^*$ large enough, we have $\max_{h\leq h_*}
\ell(h,n)=\max_{h\leq h_0-1}\ell(h,n)\vee\max_{h_0\leq h\leq h_*}\ell(h,n)\leq
\ell(h_*,n)$.

\subsection{Case $\alpha=\beta=1$}
Suppose $\alpha_n=an^{-1}$ and $1-\beta_n=b(n+1)^{-1}$, with $a>0$ and
$b>1$. For $n\geq h$, we have
\begin{equation*}
\begin{aligned}
  \ell(h+1,n)-\ell(h,n) & = \ln(a(h+1)^{-1}) -\ln(a h^{-1}) 
  -\ln( 1 - b(h+1)^{-1})
\\
& =
-\ln\Big( 1 + \frac{1}{h} \Big)-\ln\Big( 1 - \frac{b}{(h+1)}\Big) 
 = \frac{b-1}{h+1}+o(h^{-1}).  
\end{aligned}
\end{equation*}
Since $b>1$, we can argue as in the previous subsection. Therefore,
there exists $h_0$ such that the function $h\mapsto \ell(h,n)$ is
monotonically increasing for $h\geq h_0$.  Now, fix
$\eta=(b-1)/(b+1)>0$ and let $j_0$ such that $j\geq j_0$ implies
$\ln(\beta_j) \leq -\frac{b j^{-1}}{1+\eta}$.  Then take
$h_*\geq\max(h_0,j_0)+1$ and $h\leq h_0-1$. For $h_*$ large
enough, we get 
\begin{equation*}
\begin{aligned}
  \ell(h_*,n)-\ell(h,n) & =
  \ln(\alpha_{h_*}) - \ln(\alpha_{h}) -\sum_{j=h}^{h_*-1}\ln(\beta_j)
  =\ln(ah_*^{-1})-\ln(ah^{-1})-\sum_{j=h}^{h_*-1}\ln(\beta_j)
\\
&
\geq
\ln(h_*^{-1}) + \sum_{j=\max(h_0,j_0)}^{h_*-1} \frac{b j^{-1}}{1+\eta}
\\
&
\geq 
-\ln( h_*)  +
C_1 + \frac{b}{1+\eta} \int_{\max(h_0,j_0)}^{h_*-1} x^{-1} \,dx
\\
& 
= 
-\ln(h_*)  + C_1 + \frac{b}{(1+\eta)} 
\big[ \ln(h_*-1) - \ln(\max(h_0,j_0)) \big]
\\
& 
=
C_2 + \frac{b-1-\eta}{(1+\eta)} \ln(h_*)-O(1/h_*)
\\
&
= C_2 +\frac{b(b-1)}{2b}\ln(h_*)-O(1/h_*)\geq 0\,.
\end{aligned}
\end{equation*}
Therefore, taking $h^*$ large enough, we have $\max_{h\leq h_*}
\ell(h,n)=\max_{h\leq h_0-1}\ell(h,n)\vee\max_{h_0\leq h\leq h_*}\ell(h,n) \leq
\ell(h_*,n)$.

\subsection{Case $0<\alpha<\beta<(1+\alpha)/2$}
Suppose
$$
\alpha_n=
an^{-\alpha}\left(
1+\frac{c_1}{n^{1-\beta}}+
\frac{c_2}{n^{\beta-\alpha}}+\frac{c_3}{n}+O(1/n^{2-\beta})\right)
$$ and $1-\beta_n=b(n+1)^{-\beta}$, with $a,\, b>0$,
$0<\alpha<\beta<(1+\alpha)/2$ and $c_1,\,c_2,\,c_3\in \mathbb{R}$. Set
$\gamma=\beta-\alpha\in (0,1/2)$.  For $n\geq h$, we have
\begin{equation}\label{def_monot}
\begin{aligned}
  \ell(h+1,n)-\ell(h,n) & = \ln(a(h+1)^{-\alpha }) -\ln(a h^{-\alpha })
  -\ln( 1 -b(h+1)^{-\beta})\\
&  \qquad +\ln\big(
1 + c_1 /(h+1)^{1-\beta} + c_2 /(h+1)^{\gamma} + c_3 / (h+1) + O(1/h^{2-\beta})
\big)
   \\ 
&  \qquad \qquad -\ln\big(
1 + c_1/h^{1-\beta} +c_2 / h^{\gamma} + c_3/ h+ O(1/h^{2-\beta})
\big) \,.
\end{aligned}
\end{equation}
Now, we aim at obtaining a series expansion with a reminder term of
the type $o(1/h^{\beta})$.  Since $\beta<1$, the first three terms of
the right-hand side of the above equation give
\[
\ln(a(h+1)^{-\alpha }) -\ln(a h^{-\alpha })
  -\ln( 1 - b(h+1)^{-\beta}) =
  -\alpha\ln\Big( 1 + \frac{1}{h} \Big)-\ln\Big( 1 - \frac{b}{(h+1)^\beta}\Big) =
\frac{b}{(h+1)^\beta} + o(h^{-\beta}).
\]
We deal now with the last two terms of \eqref{def_monot}. We recall that
\[
\ln(1+x) =
x - \frac{x^2}{2} + \frac{x^3}{3} + \cdots +(-1)^{j-1}\frac{x^j}{j} + o(x^j)\, ,
\]
and therefore, since $2-\beta=1+1-\beta>1>\beta$ and
$j(1-\beta)>\beta$ and $j\gamma=j(\beta-\alpha)>\beta$ for $j$ large
enough, there are only a finite number $J_0$ of terms with an order
$\tau_j\leq \beta$. In other words, we can write
\begin{align*}
  &\ln\big( 1 + c_1/(h+1)^{1-\beta} + c_2 /(h+1)^{\gamma} + c_3/ (h+1)
  + O(1/n^{2-\beta})\big)
\\ 
&-\ln\big( 1 + c_1 /h^{1-\beta} +c_2 / h^{\gamma} + c_3/h+ O(1/n^{2-\beta})\big)
\\
& = \sum_{j=1}^{J_0} C_j (h+1)^{-\tau_j} - \sum_{j=1}^{J_0} C_j h^{-\tau_j}
+ o(1/h^{\beta})
\\
& 
= \sum_{j=1}^{J_0} C_j \big[ (h+1)^{-\tau_j} - h^{-\tau_j} \big] + o(1/h^{\beta})
= \sum_{j=1}^{J_0} C_j h^{-\tau_j} \big[ (1+h^{-1})^{-\tau_j} -1\big] + o(1/h^{\beta})
\\
& 
= \sum_{j=1}^{J_0} C_j h^{-\tau_j} (\tau_j h^{-1}+ o(1/h)\big) + o(1/h^{\beta})
= o(1/h^{\beta})\,.
\end{align*}
Summing up, we have 
\[
\ell(h+1,n)-\ell(h,n) = \frac{b}{(h+1)^\beta} + o(h^{-\beta}).
\]
Then there exists $h_0$ such that the function $h\mapsto \ell(h,n)$ is
monotonically increasing for $h\geq h_0$. Now, fix $\eta>0$ and let
$j_0$ such that $j\geq j_0$ implies $\ln(\beta_j) \leq -\frac{b
  j^{-\beta}}{1+\eta}$.  Then take $h_*\geq\max(h_0,j_0)+1$ and
$h\leq h_0-1$.  Since $\beta<(1+\alpha)/2$, we have
$\alpha_n=an^{-\alpha}(1+O(1/n^\gamma))$ and so, for $h_*$ large enough, we get
\begin{equation*}
\begin{aligned}
  \ell(h_*,n)-\ell(h,n) & =
  \ln(\alpha_{h_*}) - \ln(\alpha_{h}) -\sum_{j=h}^{h_*-1}\ln(\beta_j)
  \\
  &=\ln(ah_*^{-\alpha})-\ln(ah^{-\alpha})+\ln(1+O(h_*^{-\gamma}))
  +C_1-\sum_{j=h}^{h_*-1}\ln(\beta_j)
\\
&
\geq
\ln(h_*^{-\alpha})+\ln(1+O(h_*^{-\gamma})) +C_1+
\sum_{j=\max(h_0,j_0)}^{h_*-1} \frac{b j^{-\beta}}{1+\eta}
\\
&
\geq 
-\alpha \ln( h_*)  + O(h_*^{-\gamma})+C_2 + \frac{b}{1+\eta}
\int_{\max(h_0,j_0)}^{h_*-1} x^{-\beta} \,dx
\\
& 
= 
-\alpha \ln( h_*)  + O(h_*^{-\gamma}) + C_2 + \frac{b}{(1+\eta)(1-\beta)} 
\big[ (h_*-1)^{1-\beta} - \max(h_0,j_0)^{1-\beta} \big]
\\
& 
= 
C_3 + O(h_*^{-\gamma})
-\alpha \ln( h_*)  + \frac{b}{(1+\eta)(1-\beta)} (h_*-1)^{1-\beta}\geq 0\,.
\end{aligned}
\end{equation*}
Therefore, taking $h^*$ large enough, we have $\max_{h\leq h^*}
\ell(h,n)=\max_{h\leq h_0-1}\ell(h,n)\vee\max_{h_0\leq h\leq h^*}\ell(h,n)\leq
\ell(h^*,n)$.

%%%%%%%%%%%%%%%%%%%%%%%%%%%%%%%%%%%

\section{Technical results}

We recall the generalized Kronecker lemma \cite[Corollary
  A.1]{S:ale-cri-ghi-MEAN}:

\begin{lemma}\label{gen-kro-lemma} (Generalized Kronecker Lemma)\\
Let $\{v_{N,n}:1\leq n\leq N\}$ and $(z_n)_n$ be respectively a
triangular array and a sequence of complex numbers such that
 $v_{N,n}\neq 0$ and
$$
\lim_N v_{N,n}=0,\quad\lim_n v_{n,n}\;\hbox{exists finite},\quad
\sum_{n=1}^N \left|v_{N,n}-v_{N,n-1}\right|=O(1)
$$ and $\sum_n z_n$ is convergent. Then $ \lim_N \sum_{n=1}^N
v_{N,n}z_n=0.$
\end{lemma}

The above corollary is useful to get the following result for complex
random variables, which slightly extends the version provided in
\cite[Lemma A.2]{S:ale-cri-ghi-MEAN}:

\begin{lemma}\label{lemma-serie-rv}
Let ${\mathcal H}=({\mathcal H}_n)_n$ be a filtration and $(Y_n)_n$ a
$\mathcal H$-adapted sequence of complex random variables.  Moreover,
let $(c_n)_n$ be a sequence of strictly positive real numbers such
that $\sum_n E\left[|Y_n|^2\right]/c_n^2<+\infty$ and let
$\{v_{N,n},1\leq n\leq N\}$ be a triangular array of complex numbers
such that $v_{N,n}\neq 0$ and
\begin{equation*}
\lim_N v_{N,n}=0,\quad
\lim_n v_{n,n}\;\hbox{exists finite},
\quad
\sum_{n=1}^N \left|v_{N,n}-v_{N,n-1}\right|=O(1)\,.
\end{equation*}
Suppose that 
\begin{equation}\label{conv-medie-cond}
\sum_{n=1}^N v_{N,n}\frac{E[Y_n|{\mathcal H}_{n-1}]}{c_n}
\stackrel{P}\longrightarrow V,
\end{equation}
where $V$ is a suitable random variable.  Then $\sum_{n=1}^N
v_{N,n}Y_n/c_n\stackrel{P}\longrightarrow V$.  \\

\indent If the convergence in \eqref{conv-medie-cond} is almost sure,
then also the convergence of $\sum_{n=1}^N v_{N,n}Y_n/c_n$ toward
$V$ is almost sure.
\end{lemma}

\begin{proof} Consider the martingale $(M_n)_n$ defined by
$$
M_n=\sum_{j=1}^n \frac{Y_j-E[Y_j|{\mathcal H}_{j-1}]}{c_j}.
$$
It is bounded in $L^2$ since $\sum_{n} 
\frac{E[|Y_n|^2]}{c_n^2}<+\infty$ by assumption and so it is almost surely
convergent, that means $$ \sum_{n}
\frac{Y_n(\omega)-E[Y_n|{\mathcal H}_{n-1}](\omega)}{c_n}<+\infty
$$ for $\omega\in B$ with $P(B)=1$. Therefore, fixing $\omega\in B$
and setting $z_n=\frac{Y_n(\omega)-E[Y_n|{\mathcal
      H}_{n-1}](\omega)}{c_n}$, by Lemma \ref{gen-kro-lemma},
we get
$$
\lim_N \sum_{n=1}^N v_{N,n}
\frac{Y_n(\omega)-E[Y_n|{\mathcal H}_{n-1}](\omega)}{c_n}
=0\,,
$$
that is 
$$
\sum_{n=1}^N v_{N,n}
\frac{Y_n-E[Y_n|{\mathcal H}_{n-1}]}{c_n}
\stackrel{a.s.}\longrightarrow 0.
$$
In order to conclude, it is enough to observe that
$$
\sum_{n=1}^N v_{N,n}\frac{Y_n}{c_n}=
\sum_{n=1}^N v_{N,n}\frac{Y_n-E[Y_n|{\mathcal H}_{n-1}]}{c_n} +
\sum_{n=1}^N v_{N,n}\frac{E[Y_n|{\mathcal H}_{n-1}]}{c_n}
$$
and use assumption \eqref{conv-medie-cond}.
\end{proof}

\begin{remark}\label{case-as-conv} \rm 
If we have $\sum_{n=1}^N\frac{|v_{N,n}|}{c_n}=O(1)$,
$\lim_N\sum_{n=1}^N \frac{v_{N,n}}{c_n}= \lambda\in{\mathbb C}$ and
$E[Y_n|{\mathcal H}_{n-1}]\stackrel{a.s.}\longrightarrow Y$, then
\eqref{conv-medie-cond} is satisfied with almost sure convergence and
$V=\lambda Y$.  Indeed, if we denote by $A$ an event such that
$P(A)=1$ and $\lim_n E[Y_n|{\mathcal H}_{n-1}](\omega)=Y(\omega)$ for
each $\omega\in A$, then we can fix $\omega \in A$, set
$w_n=E[Y_n|{\mathcal H}_{n-1}](\omega)$ and $w=Y(\omega)$, and apply
the generalized Toeplitz lemma \cite[Lemma A.1]{S:ale-cri-ghi-MEAN}
(with $z_{N,n}=v_{N,n}/(c_n\lambda)$ and $s=1$ when $\lambda\neq 0$ and with
$z_{N,n}=v_{N,n}/c_n$ and $s=0$ when $\lambda=0$) in order to get
$\sum_{n=1}^N v_{N,n}\frac{w_n}{c_n}\to \lambda Y$ almost surely.
\end{remark}

\indent The proof of the following lemma can be found in \cite{S:del}. We here 
rewrite the proof only for the reader's convenience.
 
\begin{lemma}\label{lemma-del} (\cite{S:del}, Lemma 18)\\
Let $x_n$, $\zeta_n$, $\gamma_n$ be non-negative sequences such that 
$\gamma_n\to 0,\quad \sum_n\gamma_n=+\infty$
and
$$
x_n\leq (1-\gamma_n)x_{n-1}+\gamma_n\zeta_n.
$$
Then $\limsup_n x_n\leq \limsup_n \zeta_n$.
\end{lemma}
\begin{proof}
Take $L>\limsup_n\zeta_n$ and $n^*$ large enough 
so that $\zeta_n<L$ and $\gamma_n\leq 1$ when $n\geq n^*$. Then, using that 
$(x+y)^+\leq x^+ + y^+$, we have for $n\geq n^*$
\begin{equation*}
\begin{split}
(x_n-L)^+
&\leq 
\left((1-\gamma_n)(x_{n-1}-L)+\gamma_n(\zeta_n-L)\right)^+\\
&\leq
(1-\gamma_n)(x_{n-1}-L)^++\gamma_n(\zeta_n-L)^+\\
&\leq
(1-\gamma_n)(x_{n-1}-L)^+.
\end{split}
\end{equation*}
Since $\sum_n\gamma_n=+\infty$, the above inequality implies that
$\lim_n(x_n-L)^+=0$. This is enough to conclude, because we can choose
$L$ arbitrarily close to $\limsup_n\zeta_n$.
\end{proof}

%%%%%%%%%%%%%%%%%%%%%%%%%%%%%%

\section{Some stochastic approximation results}\label{SA-app}

Consider a stochastic process $(\bs{\theta_n})$ taking values in
$\Theta=[-1,1]^k$, adapted to a filtration
$\mathcal{F}=(\mathcal{F}_n)_n$ and following the dynamics
\begin{equation}\label{case-01}
  \bs{\theta_{n+1}} = (1-\epsilon_n)\bs{\theta_n} +
  c\epsilon_n \Delta \bs{M_{n+1}}, 
\end{equation}
where $c>0$, $(\Delta\bs{M_{n+1}})_n$ is a uniformly bounded
martingale difference sequence with respect to $\mathcal{F}$ and
$\epsilon_n=(n+1)^{-\epsilon}$ with $\epsilon\in (0,1]$ so that
$\epsilon_n\to 0$ and $\sum_n\epsilon_n=+\infty$. Setting
$\Delta\bs{ \widetilde{M}_{n+1} }=c\Delta \bs{M_{n+1}}$, equation
\eqref{case-01} becomes
\begin{equation*}%\label{case-01-bis}
  \bs{\theta_{n+1}} = (1-\epsilon_n)\bs{\theta_n} +
  \epsilon_n \Delta \bs{ \widetilde{M}_{n+1}}.
\end{equation*}
Then:

\begin{theorem}\label{SA-conv}
In the above setting, we have 
$
\bs{\theta_N}\stackrel{a.s.}\longrightarrow \bs{0}\,.
$
\end{theorem}
\begin{proof}
We have the following two cases:
\begin{itemize}
\item $\epsilon\in (1/2,1]$ so that $\sum_n\epsilon_n^2<+\infty$ or 
\item $\epsilon\in (0,1/2]$ so that $\sum_n\epsilon_n^2=+\infty$.
\end{itemize}
For the first case, we refer to \cite[Cap.~5, Th.~2.1]{S:KusYin03}. For
the second case, we refer to \cite[Cap.~5, Th.~3.1]{S:KusYin03}). In
this case, since $(\bs{\theta_n})$ and $(\Delta \bs{\widetilde{M}_n})$ are
uniformly bounded, the key assumption to be verified in order to apply
\cite[Cap.~5, Th.~3.1]{S:KusYin03} is the ``rate of change'' condition
(see \cite[p.~137]{S:KusYin03}), that is
\begin{equation*}%\label{rc-cond}
\limsup_N \sup_{t \in [0,1]} | M^0(N+t) - M^0(N)| = 0, \qquad a.s.
\end{equation*}
where $ M^0(t) = \sum_{j=0}^{m(t)-1} \epsilon_j
\Delta\bs{\widetilde{M}_{j+1}}$ and $m(t) = \inf\{ n \colon t <
t_{n+1}=\sum_{j=0}^n \epsilon_j \}$ (see
\cite[p.~122]{S:KusYin03}). Since $(\Delta\bs{\widetilde{M}_n})$ is
uniformly bounded, the above condition is satisfied when the following
simpler conditions are satisfied (see \cite[pp. 139-141]{S:KusYin03}):
\begin{itemize}
\item[(i)] For each $u>0$ $\sum_n e^{-u /\epsilon_n}<+\infty$;
\item[(ii)] For some $T<+\infty$, there exists a constant
  $c(T)<+\infty$ such that $ \sup_{n\leq j\leq
  m(t_n+T)}\frac{\epsilon_j}{\epsilon_n}\leq c(T)$. 
\end{itemize}
When $\epsilon_n=(1+n)^{-\epsilon}$, condition (i) is obviously
verified, because we have $\lim_n n^2/e^{u(1+n)^{-\epsilon}}=0$.
Finally, condition (ii) is always satisfied when $\epsilon_n$ is
decreasing, as it is in the case
$\epsilon_n=(1+n)^{-\epsilon}$. Indeed, we simply have $ \sup_{n\leq
  j\leq m(t_n+T)}\epsilon_j/\epsilon_n=\epsilon_n/\epsilon_n=1$.
\end{proof}

\begin{theorem}\label{SA-CLT}
In the above setting, if we have
$E[\Delta\bs{M_{n+1}}\Delta\bs{M_{n+1}}^{\top}|\mathcal{F}_n]
\stackrel{a.s.}\longrightarrow\Gamma$ with $\Gamma$ a symmetric
positive definite matrix, then we have
$$
\frac{1}{\sqrt{\epsilon_N}}\bs{\theta_N}\stackrel{d}\longrightarrow
\mathcal{N}(\bs{0},\Sigma),
$$ where $\Sigma=c^2\Gamma/2$ when $\epsilon\in (0,1)$ and
$\Sigma=c^2\Gamma$ when $\epsilon=1$.
\end{theorem}
\begin{proof} We 
have $\bs{\theta_N}\stackrel{a.s.}\longrightarrow \bs{0}$ and $\bs{0}$
belongs to the interior part of $\Theta$.  Moreover, we have
$$E[\Delta\bs{\widetilde{M}_{n+1}}\Delta\bs{\widetilde{M}_{n+1}}^{\top}|
  \mathcal{F}_n] \stackrel{a.s.}\longrightarrow c^2\Gamma.$$ For the case
$\epsilon\in (1/2,1]$, we refer to \cite[Th.~2.1]{S:fort} (with $h=Id$,
  $U_*=c^2\Gamma$ and $\gamma_*=1$) and \cite[Th.~1]{S:pel} (with
  $H=-Id$, $\gamma_n=\sigma_n=\epsilon_n$ and so $\gamma_0=1$ and
  $\beta=\epsilon$).  For the case $\epsilon\in (0,1/2]$, we refer to
    \cite[cap.10, Th.~2.1]{S:KusYin03} (with $A=-Id$).  The key
    assumption for applying this theorem is
    $\bs{\theta_n}/\sqrt{\epsilon_n}$ tight. On the other hand, in the
    considered setting, this last condition is satisfied because of
    \cite[Th.~4.1]{S:KusYin03}. Note that the limit distribution
    corresponds to the stationary distribution of the diffusion
$$
dU_t= \left(-Id+c(\epsilon)\right) U_t dt+ c\Gamma^{1/2} dW_t,
$$
where $W=(W_t)_t$ is a standard Wiener process and 
\begin{equation*}
c(\epsilon)=\begin{cases}
0\quad\hbox{for } \epsilon<1\\
1/2\quad\hbox{for } \epsilon=1.
\end{cases}
\end{equation*}
Therefore the limit covariance matrix is determined by solving the
associated Lyapunov's equation \cite{S:pel}, that, in the considered case, 
simply is
$$
2\left(-Id+c(\epsilon)Id\right)\Sigma=-c^2\Gamma.
$$
\end{proof}

\begin{theorem}\label{SA-CLT-mu} 
In the above setting, let $(\bs{\mu_n})$ be another stochastic process
taking values in $\Theta=[-1,1]^k$, adapted to a filtration
$\mathcal{F}$ and following the dynamics
\begin{equation*}
\bs{\mu_{n+1}}-\bs{\mu_n}=-\frac{1}{n}(\bs{\mu_n}-\bs{\theta_n})
+\frac{1}{n}\Delta\bs{M_{n+1}}\,.
\end{equation*}
Suppose that
$E[\Delta\bs{M_{n+1}}\Delta\bs{M_{n+1}}^{\top}|\mathcal{F}_n]
\stackrel{a.s.}\longrightarrow \Gamma$. If $\epsilon\in (1/2,1)$, then
  we have
$$
\begin{pmatrix}
 \sqrt{N}\bs{\mu}_N\\
 \epsilon_N^{-1/2}\bs{\theta}_N
\end{pmatrix}
\stackrel{d}\longrightarrow
\mathcal{N}\left(\bs{0},\,
\begin{pmatrix}
(c+1)^2\Gamma & \bs{0}\\
 \bs{0}  & \frac{c^2}{2}\Gamma
\end{pmatrix}
\right).
$$
If $\epsilon=1$, then we have
$$
\begin{pmatrix}
 \sqrt{N}\bs{\mu}_N\\
 \epsilon_N^{-1/2}\bs{\theta}_N
\end{pmatrix}
\stackrel{d}\longrightarrow
\mathcal{N}\left(\bs{0},\,
\begin{pmatrix}
[(c+1)^2+c^2]\Gamma & c(c+1)\Gamma\\
 c(c+1)\Gamma  & c^2\Gamma
\end{pmatrix}
\right).
$$
\end{theorem}
\begin{proof} 
The dynamics for the pair $(\bs{\mu_n},\bs{\theta_n})_n$ is 
\begin{equation*}
\begin{cases}
\bs{\mu_{n+1}}-\bs{\mu_n}&=-\frac{1}{n}(\bs{\mu_n}-\bs{\theta_n})
+\frac{1}{n}\Delta\bs{M_{n+1}}\\
\bs{\theta_{n+1}}-\bs{\theta_{n}}&=-\epsilon_{n}\bs{\theta_{n}}
+c\epsilon_n\Delta\bs{M_{n+1}}=-\epsilon_{n}\bs{\theta_{n}}
+\epsilon_n\Delta\bs{\widetilde{M}_{n+1}}\,.
\end{cases}
\end{equation*}
with $E[\Delta\bs{M_{n+1}}\Delta\bs{M_{n+1}}^{\top}\,|\mathcal{F}_n]
\stackrel{a.s.}\longrightarrow \Gamma$.  Therefore, when
$1/2<\epsilon<1$, the statement follows from \cite{S:mok-pel2006} (with
$Q_{11}=Q_{22}=-Id$, $Q_{12}=Id$, $Q_{21}=\bs{0}$, $b=\beta_0=1$,
$a=\epsilon$, $\Gamma_{11}=\Gamma$, $\Gamma_{22}=c^2\Gamma$ and
$\Gamma_{12}=\Gamma_{21}=c\Gamma$). In particular, the two blocks of the limit
covariance matrix, say $\Sigma_{\mu}$ and $\Sigma_{\theta}$, are
determined solving the equations
$$
(H+\frac{1}{2}Id)\Sigma_{\mu}+\Sigma_\mu(H^{\top}+\frac{1}{2}Id)=-\Gamma_{\mu},
$$ where $H=Q_{11}-Q_{12}Q_{22}^{-1}Q_{21} = - Id + \bs{0}$ and
$\Gamma_{\mu}=\Gamma_{11}+Q_{12}Q_{22}^{-1}\Gamma_{22}(Q_{22}^{-1})^{\top}Q_{12}^{\top}-
\Gamma_{12}(Q_{22}^{-1})^{\top}Q_{12}^{\top}-Q_{12}Q_{22}^{-1}\Gamma_{21} = 
\Gamma +c^2\Gamma + c\Gamma + c \Gamma = (c+1)^2\Gamma$,
and
$$
Q_{22}\Sigma_\theta+\Sigma_\theta Q_{22}^{\top}=-\Gamma_{22}.
$$

When $\epsilon=1$, we can conclude by \cite{S:pel} or \cite{S:Zhang2016}
taking $\bs{X_n}=(\bs{\mu_n},\,\bs{\theta_n})^{\top}$. Indeed, in this
case the covariance matrix is given by
$$
(H+\frac{1}{2}Id)\Sigma+\Sigma(H^\top+\frac{1}{2}Id)=-\widetilde{\Gamma},
$$
where
$$
H=\begin{pmatrix}
-Id & Id\\
 \bs{0}  & -Id
\end{pmatrix}
\quad\mbox{and}\quad
\widetilde{\Gamma}=\begin{pmatrix}
\Gamma & c\Gamma\\
c\Gamma  & c^2\Gamma
\end{pmatrix}.
$$ Therefore, if we split $\Sigma$ in blocks, say $\Sigma_{\mu}$,
$\Sigma_\theta$ and $\Sigma_{\mu\theta}$, we find the system
\begin{equation*}
  \begin{split}
-\Sigma_\mu+2\Sigma_{\mu\theta}&=-\Gamma\\
-\Sigma_{\mu\theta}+\Sigma_{\theta}&=-c\Gamma\\
-\Sigma_\theta&=-c^2\Gamma
  \end{split}
\end{equation*}
and so the proof is concluded by solving this system.
\end{proof}

%%%%%%%%%%%%%%%%%%%%%%%%%%%%%
\section{Stable convergence}\label{stable-conv-app}

This brief section contains some basic definitions and results
concerning stable convergence. For more details, we refer the reader
to \cite{S:crimaldi-libro,S:CriLetPra,S:HallHeyde} and the references
therein.\\

\indent Let $(\Omega, {\mathcal A}, P)$ be a probability space, and
let $S$ be a Polish space, endowed with its Borel $\sigma$-field. A
{\em kernel} on $S$, or a random probability measure on $S$, is a
collection $K=\{K(\omega):\, \omega\in\Omega\}$ of probability
measures on the Borel $\sigma$-field of $S$ such that, for each
bounded Borel real function $f$ on $S$, the map
$$
\omega\mapsto K\!f(\omega)=\int f (x)\, K(\omega)(dx)
$$ is $\mathcal A$-measurable. Given a sub-$\sigma$-field $\mathcal H$
of $\mathcal A$, a kernel $K$ is said $\mathcal H$-measurable if all
the above random variables $K\!f$ are $\mathcal H$-measurable. A
probability measure $\nu$ can be identified with a constant kernel
$K(\omega)=\nu$ for each $\omega$.\\

\indent On $(\Omega, {\mathcal A},P)$, let $(Y_n)_n$ be a sequence
of $S$-valued random variables, let $\mathcal H$ be a
sub-$\sigma$-field of $\mathcal A$, and let $K$ be a $\mathcal
H$-measurable kernel on $S$. Then, we say that $Y_n$ converges {\em
$\mathcal H$-stably} to $K$, and we write $Y_n\longrightarrow K$
${\mathcal H}$-stably, if
$$
P(Y_n \in \cdot \,|\, H)\stackrel{weakly}\longrightarrow
E\left[K(\cdot)\,|\, H \right] \qquad\hbox{for all } H\in{\mathcal
H}\; \hbox{with } P(H) > 0,
$$where $K(\cdot)$ denotes the random variable defined, for each Borel
set $B$ of $S$, as $\omega\mapsto K\!I_B(\omega)=K(\omega)(B)$.  In
the case when ${\mathcal H}={\mathcal A}$, we simply say that $Y_n$
converges {\em stably} to $K$ and we write $Y_n\longrightarrow K$
stably. Clearly, if $Y_n\longrightarrow K$ ${\mathcal H}$-stably, then
$Y_n$ converges in distribution to the probability distribution
$E[K(\cdot)]$. The $\mathcal H$-stable convergence of $Y_n$ to $K$ can
be stated in terms of the following convergence of conditional
expectations:
\begin{equation}\label{def-stable}
E[f(Y_n)\,|\, {\mathcal H}]\stackrel{\sigma(L^1,\, L^{\infty})}\longrightarrow
K\!f
\end{equation}
for each bounded continuous real function $f$ on $S$. In
\cite{S:CriLetPra} the notion of $\mathcal H$-stable convergence
is firstly generalized in a natural way replacing in
(\ref{def-stable}) the single sub-$\sigma$-field $\mathcal H$ by a
collection ${\mathcal G}=({\mathcal G}_n)$ (called conditioning
system) of sub-$\sigma$-fields of $\mathcal A$ and then it is
strengthened by substituting the convergence in
$\sigma(L^1,L^{\infty})$ by the one in probability (i.e. in $L^1$,
since $f$ is bounded). Hence, according to \cite{S:CriLetPra}, we
say that $Y_n$ converges to $K$ {\em stably in the strong sense}, with
respect to ${\mathcal G}=({\mathcal G}_n)$, if
\begin{equation*}%\label{def-stable-strong}
E\left[f(Y_n)\,|\,{\mathcal G}_n\right]\stackrel{P}\longrightarrow K\!f
\end{equation*}
for each bounded continuous real function $f$ on $S$.
\\

\indent We now conclude this section recalling some convergence
results that we apply in our proofs.  \\

\indent From \cite[Th.~3.2]{S:HallHeyde} (see also \cite[Th.~5 and
  Cor.~7]{S:CriLetPra} or \cite[Th.~5.5.1 and
  Cor.~5.5.2]{S:crimaldi-libro}), we get:

\begin{theorem}\label{stable-conv} 
Given a filtration $\mathcal{F}=(\mathcal{F}_n)_n$, let
$(\bs{Y_{N,n}})_{N,n}$ be a triangular array of random variables with
values in $\mathbb{R}^k$ such that $Y_{N,n}$ is
$\mathcal{F}_n$-measurable and
$E[\bs{Y_{N,n}}|\mathcal{F}_{n-1}]=\bs{0}$. Suppose that the following
two conditions are satisfied:
\begin{itemize}
\item[(c1)] $E\left[\,\max_{1\leq n\leq N} |\bs{Y_{N,n}} |\,\right]\to 0$ and 
\item[(c2)] $\sum_{n=1}^N \bs{Y_{N,n}} {\bs{Y_{N,n}}^\top}
\stackrel{P}\longrightarrow V$, where $V$ is
  a random variable with values in the space of positive semidefinite
  $k\times k$-matrices.
\end{itemize}
Then $\sum_{n=1}^N\bs{Y_{N,n}}$ converges stably to the Gaussian
kernel $\mathcal{N}(\bs{0},V)$. 
\end{theorem}

From \cite[Th.~5, Cor.~7, Rem.~4]{S:CriLetPra} or \cite[Th.~5.5.1,
  Cor.~5.5.2, Rem.~5.5.2]{S:crimaldi-libro}), we obtain:

\begin{theorem}\label{strong-stable-conv}
Let $(\bs{L_n})$ be a $\mathbb{R}^k$-valued martingale with respect to the
filtration $\mathcal{F}=(\mathcal{F}_n)$. Suppose that
$\bs{L_n}\stackrel{a.s.,\, L^1}\longrightarrow \bs{L}$ for some
$\mathbb{R}^k$-valued random variable $L$ and
\begin{itemize}
\item[(c1)] $n^e\,E[\sup_{j\geq n} |\bs{L_{j-1}}-\bs{L_j}|\, ]
  \longrightarrow 0$ and
\item[(c2)] $n^{2e}\sum_{j\geq n}(\bs{L_{j-1}}-\bs{L_j})(\bs{L_{j-1}}-\bs{L_j})^{\top}
  \stackrel{P}\longrightarrow V$, where $V$ is
  a random variable with values in the space of positive semidefinite
  $k\times k$-matrices.
\end{itemize}
Then
$$n^e\,\bigl(\bs{L_n}-\bs{L}\bigr)
\longrightarrow\mathcal{N}(0,V)\quad\mbox{stably
in strong sense w.r.t. } \mathcal{F}.
$$
\end{theorem}
Indeed, following \cite[Example~6]{S:CriLetPra}, it is enough to observe
that $\bs{L_n}-\bs{L}$ can be written as $\bs{L_n}-\bs{L}=\sum_{j\geq
  n}(\bs{L_j}-\bs{L_{j+1}})$.
\\
  
\indent Finally, the following result combines together a stable
convergence and a stable convergence in the strong sense \cite[Lemma
  1]{S:BeCrPrRi11}.

\begin{theorem}\label{blocco}
Suppose that $C_n$ and $D_n$ are $S$-valued random variables, that $M$
and $N$ are kernels on $S$, and that ${\mathcal G}=({\mathcal G}_n)_n$
is an (increasing) filtration satisfying for all $n$
$$
\sigma(C_n)\underline\subset{\mathcal G}_n\quad\hbox{and }\quad
\sigma(D_n)\underline\subset
\sigma\left({\textstyle\bigcup_n}{\mathcal G}_n\right)\,.
$$
\noindent If $C_n$ stably converges to $M$ and $D_n$ converges to $N$
stably in the strong sense, with respect to $\mathcal G$, then
$$
[C_n, D_n]\longrightarrow M \otimes N\qquad stably.
$$
(Here, $M\otimes N$ is the kernel on $S\times S$ such that $(M
\otimes N )(\omega) = M(\omega) \otimes N(\omega)$ for all $\omega$.)
\end{theorem}
This last result contains as a special case the fact that stable
convergence and convergence in probability combine well: that is, if
$C_n$ stably converges to $M$ and $D_n$ converges in probability to a
random variable $D$, then $(C_n,D_n)$ stably converges to $M \otimes
\delta_D$, where $\delta_D$ denotes the Dirac kernel concentrated in $D$.

\end{document}

%% file: Table_1.tex
% latex table generated in R 4.0.3 by xtable 1.8-4 package
% Tue Dec 15 15:14:07 2020
\begin{table}[htbp]
\centering
\begin{tabular}{lrrrrrrrr}
  \hline
Date & $\mathrm{Obs}_+$ & $\mathrm{Obs}_-$ & $\mathrm{Exp}_+$ & $\mathrm{Exp}_-$ & $\chi^2_+$ & $\chi^2_-$ & ${\chi^2_+}^{(c)}$ & ${\chi^2_-}^{(c)}$ \\ 
  \hline
2020-02-20 &  25 &  43 & 35.11 & 32.89 & 2.91 & 3.11 & 0.46 & 0.49 \\ 
  2020-02-23 & 53564 & 60476 & 58886.18 & 55153.82 & 481.02 & 513.58 & 2.99 & 3.19 \\ 
  2020-02-26 & 29831 & 37175 & 34599.51 & 32406.49 & 657.20 & 701.67 & 5.15 & 5.50 \\ 
  2020-02-29 & 18220 & 22184 & 20863.18 & 19540.82 & 334.87 & 357.53 & 3.27 & 3.49 \\ 
  2020-03-03 & 16801 & 14834 & 16335.18 & 15299.82 & 13.28 & 14.18 & 0.14 & 0.15 \\ 
  2020-03-06 & 27906 & 27030 & 28366.99 & 26569.01 & 7.49 & 8.00 & 0.06 & 0.07 \\ 
  2020-03-09 & 41650 & 34769 & 39460.04 & 36958.96 & 121.54 & 129.76 & 0.90 & 0.96 \\ 
  2020-03-12 & 255 & 156 & 212.23 & 198.77 & 8.62 & 9.20 & 0.62 & 0.67 \\ 
  2020-03-15 & 14193 & 13562 & 14331.69 & 13423.31 & 1.34 & 1.43 & 0.02 & 0.02 \\ 
  2020-03-18 & 12064 & 10089 & 11439.02 & 10713.98 & 34.15 & 36.46 & 0.43 & 0.46 \\ 
  2020-03-21 & 11571 & 10026 & 11151.92 & 10445.08 & 15.75 & 16.81 & 0.20 & 0.22 \\ 
  2020-03-24 & 13339 & 9172 & 11623.88 & 10887.12 & 253.07 & 270.20 & 3.19 & 3.41 \\ 
  2020-03-27 & 14798 & 10039 & 12824.94 & 12012.06 & 303.55 & 324.09 & 3.67 & 3.92 \\ 
  2020-03-30 & 12689 & 10651 & 12051.94 & 11288.06 & 33.67 & 35.95 & 0.42 & 0.45 \\ 
  2020-04-02 & 12714 & 9300 & 11367.24 & 10646.76 & 159.56 & 170.36 & 2.03 & 2.17 \\ 
  2020-04-05 & 13373 & 10815 & 12489.82 & 11698.18 & 62.45 & 66.68 & 0.76 & 0.82 \\ 
  2020-04-08 & 14889 & 11987 & 13877.81 & 12998.19 & 73.68 & 78.67 & 0.86 & 0.92 \\ 
  2020-04-11 & 12153 & 10777 & 11840.23 & 11089.77 & 8.26 & 8.82 & 0.10 & 0.11 \\ 
  2020-04-14 & 13406 & 11430 & 12824.42 & 12011.58 & 26.37 & 28.16 & 0.32 & 0.34 \\ 
  2020-04-17 & 13977 & 11371 & 13088.80 & 12259.20 & 60.27 & 64.35 & 0.72 & 0.77 \\ 
  2020-04-20 & 13753 & 12393 & 13500.86 & 12645.14 & 4.71 & 5.03 & 0.06 & 0.06 \\ 
   \hline
\end{tabular}
\caption{Contingency table associated to COVID-Twitter data: $\mathrm{Obs}_+$
  ($\mathrm{Obs}_-$) are the number of posts with positive (negative)
  sentiment posted in the day $\ell$ reported in the first column
  (DataTime); $\mathrm{Exp}_+$ ($\mathrm{Exp}_-$) corresponds to $N_\ell p_0^*$
  (resp. $N_\ell(1-p_0^*)$), where $N_\ell=\mathrm{Obs}_+ + \mathrm{Obs}_-$;
  $\chi^2_+$ ($\chi^2_-$) is the quantity
  ($\mathrm{Obs}_+-\mathrm{Exp}_+)^2/\mathrm{Exp}_+$
  (resp. $(\mathrm{Obs}_--\mathrm{Exp}_-)^2/\mathrm{Exp}_-$); ${\chi^2_+}^{(c)}$
  (${\chi^2_-}^{(c)}$) is the quantity $\chi^2_+/N_\ell^{\widehat{\eta}}$
  (resp. $\chi^2_-/N_\ell^{\widehat{\eta}}$). The statistics $Q_\ell$
  corresponds to ${\chi^2_+}^{(c)}+{\chi^2_-}^{(c)}$.}
\label{table-contingency}
\end{table}

%% file: TestBox2.tex
% latex table generated in R 4.0.3 by xtable 1.8-4 package
% Thu Dec 17 10:45:04 2020
\begin{table}[ht]
\centering
\begin{tabular}{|l|rrrrrrrrrr|}
   \hline
Df & 1 & 2 & 3 & 4 & 5 & 6 & 7 & 8 & 9 & 10 \\ 
$\chi^2$  & 3.454 & 3.624 & 4.209 & 4.640 & 5.065 & 7.103 & 8.660 & 8.812 & 10.360 & 12.852 \\ 
$p\mathrm{-value}$ & 0.063 & 0.163 & 0.240 & 0.326 & 0.408 & 0.311 & 0.278 & 0.358 & 0.322 & 0.232 \\ 
   \hline
\end{tabular}
\caption{Summary of Ljung–Box test for autocorrelation of $\{Q_\ell:\,
  \ell=1,\dots, 21\}$, with different numbers of autocorrelation lags
  being tested.  Df: number of lags under investigation; $\chi^2$:
  Ljung–Box test statistics, which is distributed as a $\chi^2$
  distribution with Df degrees of freedom under the null hypothesis of
  independence; $p\mathrm{-value}$: $p\mathrm{-value}$ of the
  Ljung–Box test.\\ The strong emotional involvement of the considered
  period had a ``mixing effect'' that cancelled possible significant
  autocorrelation during different 3-delayed days.}
\label{tab:LBtest}
\end{table}